\documentclass[a4paper,11pt,reqno]{amsart}

\usepackage[margin=1in,includehead,includefoot]{geometry}
\usepackage[utf8]{inputenc}
\usepackage[T1]{fontenc}
\usepackage{lmodern,microtype}
\usepackage{upref}
\usepackage{mathtools,amssymb,amsthm}
\usepackage[foot]{amsaddr}
\usepackage{enumitem}
\usepackage{xifthen}
\usepackage[textsize=footnotesize]{todonotes}
\usepackage{tikz}
\usepackage{graphicx}
\usepackage{caption}
\usepackage{wrapfig}
\usepackage{hyperref}
\usepackage[b]{esvect} 
\usepackage{textcomp}
\usepackage{mdframed}
\usepackage{xpatch}
\usepackage{xcolor}
\usepackage{bookmark}
\usepackage{mycommands}

\newif\iffull  
\fulltrue
\newif\ifnotfull  
\iffull\notfullfalse\else\notfulltrue\fi  

\graphicspath{{./graphics/}}

\makeatletter
\let\old@setaddresses\@setaddresses
\def\@setaddresses{\bigskip{\parindent 0pt\let\scshape\relax\let\ttfamily\relax\old@setaddresses}}
\makeatother

\tikzstyle{vertex}=[draw,circle,fill,inner sep=1pt]
\tikzstyle{maxarc}=[color=red, very thick]
\tikzstyle{invisible}=[color=white, opacity=0]

\newtheorem{theorem}{Theorem}
\newtheorem{corollary}[theorem]{Corollary}

\newtheorem{lemma}[theorem]{Lemma}
\theoremstyle{remark}

\newtheorem{remark}[theorem]{Remark}


\hypersetup{
  pdftitle={Combinatorial generation via permutation languages. II. Lattice congruences},
  pdfauthor={Hung Phuc Hoang, Torsten M\"utze}
}

\title{Combinatorial generation via permutation languages. \\ II. Lattice congruences}

\author{Hung P. Hoang}
\address[Hung P. Hoang]{Department of Computer Science, ETH Z\"urich, Switzerland}
\email{hung.hoang@inf.ethz.ch}

\author{Torsten M\"utze}
\address[Torsten M\"utze]{Department of Computer Science, University of Warwick, United Kingdom}
\email{torsten.mutze@warwick.ac.uk}

\thanks{An extended abstract of this paper appeared in the Proceedings of the 31st Annual ACM-SIAM Symposium on Discrete Algorithms (SODA~2020)~\cite{MR4141256}.}
\thanks{Torsten M\"utze is also affiliated with the Faculty of Mathematics and Physics, Charles University Prague, Czech Republic. He was supported by Czech Science Foundation grant GA~19-08554S, and by German Science Foundation grant~413902284.}

\begin{document}

\begin{abstract}
This paper deals with lattice congruences of the weak order on the symmetric group, and initiates the investigation of the cover graphs of the corresponding lattice quotients.
These graphs also arise as the skeleta of the so-called quotientopes, a family of polytopes recently introduced by Pilaud and Santos \textit{[Bull.~Lond.~Math.~Soc., 51:406–420, 2019]}, which generalize permutahedra, associahedra, hypercubes and several other polytopes.
We prove that all of these graphs have a Hamilton path, which can be computed by a simple greedy algorithm.
This is an application of our framework for exhaustively generating various classes of combinatorial objects by encoding them as permutations.
We also characterize which of these graphs are vertex-transitive or regular via their arc diagrams, give corresponding precise and asymptotic counting results, and we determine their minimum and maximum degrees.
\iffull
Moreover, we investigate the relation between lattice congruences of the weak order and pattern-avoiding permutations.
\fi
\end{abstract}

\maketitle

\captionsetup{width=.92\linewidth}

\section{Introduction}
\label{sec:intro}

We let $S_n$ denote the set of all permutations on the set $\{1,\ldots,n\}$.
The \emph{inversion set} of a permutation $\pi\in S_n$ is the set of all decreasing pairs of values of~$\pi=a_1\cdots a_n$, formally
\begin{equation*}
\inv(\pi):=\big\{(a_i,a_j)\mid 1\leq i<j\leq n\text{ and } a_i>a_j\big\}.
\end{equation*}
We consider the classical \emph{weak order} on~$S_n$, the poset obtained by ordering all permutations from~$S_n$ by containment of their inversion sets, i.e., $\pi<\rho$ for any two permutations~$\pi,\rho$ in the weak order if and only if~$\inv(\pi)\seq \inv(\rho)$; see the left hand side of Figure~\ref{fig:cong}.
Equivalently, the weak order on~$S_n$ can be obtained as the poset of regions of the \emph{braid arrangement} of hyperplanes.
Also, its Hasse diagram is the graph of the \emph{permutahedron}.

It is well-known that the weak order forms a \emph{lattice}, i.e., joins $\pi\vee\rho$ and meets $\pi\wedge\rho$ are well-defined.
A \emph{lattice congruence} is an equivalence relation~$\equiv$ on~$S_n$ that is compatible with taking joins and meets.
Formally, if~$\pi\equiv\pi'$ and~$\rho\equiv\rho'$ then we also have $\pi\vee \rho\equiv \pi'\vee \rho'$ and $\pi\wedge \rho\equiv \pi'\wedge \rho'$.
The \emph{lattice quotient}~$S_n/{\equiv}$ is obtained by taking the equivalence classes as elements, and ordering them by~$X<Y$ if and only if there is a~representative $\pi\in X$ and a representative~$\rho\in Y$ such that~$\pi<\rho$ in the weak order; see the right hand side of Figure~\ref{fig:cong}.
The study of lattice congruences of the weak order has been developed considerably in recent years, in particular thanks to Reading's works, summarized in~\cite{MR3221544, MR3645056, MR3645055}.
All of these results have beautiful ramifications into posets, polytopes, geometry, and combinatorics.
In fact, many of these results even hold in the more general setting of arbitrary Coxeter groups and for the poset of regions of general hyperplane arrangements.

It is not hard to see that there are double-exponentially (in $n$) many distinct lattice congruences of the weak order on~$S_n$, and many important lattices arise as quotients of suitable lattice congruences:
the Boolean lattice, the Tamari lattice~\cite{MR0146227} (shown in Figure~\ref{fig:cong}), type~A Cambrian lattices~\cite{MR2258260,MR3628225}, permutree lattices~\cite{MR3856522}, the increasing flip lattice on acyclic twists~\cite{MR3741436}, and the rotation lattice on diagonal rectangulations~\cite{MR2871762,MR2914637,MR3878132}.

In a recent paper, Pilaud and Santos~\cite{MR3964495} showed how to realize the cover graph of any lattice quotient~$S_n/{\equiv}$ as the graph of an $(n-1)$-dimensional polytope, and they called these polytopes \emph{quotientopes}.
Their results generalize many earlier constructions of polytopes for the aforementioned special lattices~\cite{MR2108555,MR2321739,MR3800847,MR3856522,MR2864447,MR2871762}.
In particular, quotientopes generalize permutahedra, associahedra, and hypercubes.
Interestingly, quotientopes are defined by a set of gliding hyperplanes that is consistent with refining the corresponding lattice congruences, i.e., moving the hyperplanes outwards corresponds to refining the equivalence classes.
In particular, the permutahedron contains all other quotientopes, and the hypercube is contained in all (full-dimensional) quotientopes.
Figure~\ref{fig:quotient} shows all quotientopes for $n=4$ ordered by refinement of the corresponding congruences, with permutahedron, associahedron, and 3-cube highlighted.

\begin{figure}
\includegraphics{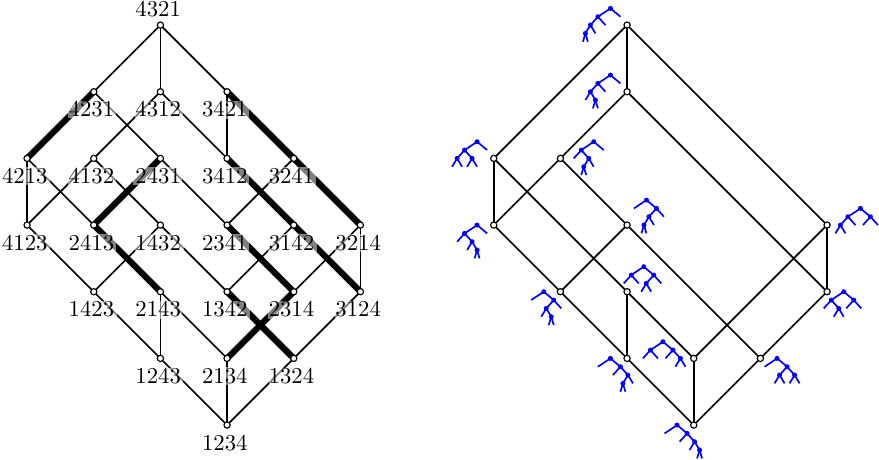}
\caption{Hasse diagrams of the weak order on~$S_4$ (left) with a lattice congruence~$\equiv$ (bold edges), and of the resulting lattice quotient~$S_4/{\equiv}$ (right), which is the well-known Tamari lattice (with corresponding binary trees).
}
\label{fig:cong}
\end{figure}

There are several long-standing open problems revolving around Hamilton paths and cycles in graphs of polytopes and other highly symmetric graphs, most prominently Barnette's conjecture and Lov{\'a}sz' conjecture.
Barnette's conjecture~\cite{MR0250896} asserts that the graph of every simple three-dimensional polytope with an even number of edges on each face has a Hamilton cycle.
Another variant of the conjecture states that the graphs of all simple three-dimensional polytopes with face sizes at most~6, in particular all fullerenes, have a Hamilton cycle~\cite{MR1737931}.
Barnette also conjectured that the graph of every simple 4-dimensional polytope has a Hamilton cycle~\cite[p.~1145]{MR266050}.
Note that the simplicity of these polytopes means that their graphs are 3-regular or 4-regular, respectively.
Lov{\'a}sz' conjecture~\cite{MR0263646} asserts that every vertex-transitive graph has a Hamilton path.
A stronger form of his conjecture asserts that such graphs even have a Hamilton cycle, with five well-understood exceptions, among them the Petersen graph and the Coxeter graph.

\subsection{Our results}

In this paper we initiate the investigation of the cover graphs of lattice quotients of the weak order on the symmetric group~$S_n$, or equivalently, of the graphs of the quotientopes introduced by Pilaud and Santos.
Our first main result is that for every lattice congruence~$\equiv$ of the weak order on~$S_n$, the cover graph of the lattice quotient $S_n/{\equiv}$ has a Hamilton path (Theorem~\ref{thm:lattice}).
As a consequence, the graph of every quotientope has a Hamilton path (Corollary~\ref{cor:quotient}); see Figure~\ref{fig:quotient} on page~\pageref{fig:quotient}.
These Hamilton paths are computed by a simple greedy algorithm, which we devised in~\cite{MR4391718} within a general framework for exhaustively generating various classes of combinatorial objects by encoding them as permutations.
For the permutahedron, associahedron, and hypercube, algorithmic constructions of such Hamilton paths were already known by the Steinhaus-Johnson-Trotter algorithm~\cite{DBLP:journals/cacm/Trotter62,MR0159764}, by the Lucas-Roelants van Baronaigien-Ruskey tree rotation algorithm~\cite{MR1239499} (see also \cite{MR920505,MR1723053}), and by the binary reflected Gray code~\cite{gray_1953}, respectively.
Our results thus unify and generalize all these classical algorithms.
Motivated by our Hamiltonicity results and by Barnette's and Lov{\'a}sz' conjectures, we also characterize which lattice congruences of the weak order on~$S_n$ yield regular or vertex-transitive quotientopes.
This characterization uses arc diagrams introduced by Reading~\cite{MR3335492}, and allows us to derive corresponding precise and asymptotic counting results.
We also determine the minimum and maximum degrees of quotientopes.
All of these results are summarized in Table~\ref{tab:trans} on page~\pageref{tab:trans} (theorems are referenced in the table).
In those results, Catalan numbers, integer compositions and partitions, and the Erd\H{o}s-Szekeres theorem make their appearance.
\iffull
As a last result, we formulate conditions under which a set of pattern-avoiding permutations can be realized as a lattice congruence of the weak order on~$S_n$ (Theorem~\ref{thm:well} on page~\pageref{thm:well}).
\fi

\subsection{Outline of this paper}

This is part~II in our paper series on exhaustively generating various classes of combinatorial objects by encoding them as permutations.
In part~I~\cite{MR4391718}, we developed the fundamentals of this framework, including a simple greedy algorithm for exhaustive generation, and we applied the framework to generate many different classes of pattern-avoiding permutations.
In Section~\ref{sec:recap} of the present paper, we briefly recap the necessary background from this first paper.
In Section~\ref{sec:cong}, we apply our framework to generating lattice congruences of the weak order on~$S_n$, proving that all quotientopes have a Hamilton path.
In Section~\ref{sec:reg}, we characterize and count regular and vertex-transitive quotientopes, and we determine their minimum and maximum degree.
\iffull
Lastly, in Section~\ref{sec:pattern}, we briefly discuss the relation between pattern-avoiding permutations and lattice congruences of the weak order.
\fi
We conclude with some interesting open problems in Section~\ref{sec:open}.

\section{Recap: Zigzag languages and Algorithm~J}
\label{sec:recap}

\subsection{Preliminaries}

For any two integers~$a$ and~$b$ with~$a\leq b$ we define $[a,b]:=\{a,a+1,\ldots,b\}$ and $\left]a,b\right[:=[a,b]\setminus\{a,b\}$, and we introduce the abbreviation $[n]:=[1,n]$.
We use  $\ide_n=12\cdots n$ to denote the identity permutation, and $\varepsilon\in S_0$ to denote the empty permutation.
For any $\pi\in S_{n-1}$ and any $1\leq i\leq n$, we write $c_i(\pi)\in S_n$ for the permutation obtained from~$\pi$ by inserting the new largest value~$n$ at position~$i$ of~$\pi$, i.e., if $\pi=a_1\cdots a_{n-1}$ then $c_i(\pi)=a_1\cdots a_{i-1} \, n\, a_i \cdots a_{n-1}$.
Moreover, for~$\pi\in S_n$, we write $p(\pi)\in S_{n-1}$ for the permutation obtained from~$\pi$ by removing the largest entry~$n$.

Given a permutation $\pi=a_1\cdots a_n$ with a substring $a_i\cdots a_j$ with $a_i>a_{i+1},\ldots,a_j$, a \emph{right jump of the value~$a_i$ by $j-i$~steps} is a cyclic left rotation of this substring by one position to $a_{i+1}\cdots a_j a_i$.
Similarly, given a substring $a_i\cdots a_j$ with $a_j>a_i,\ldots,a_{j-1}$, a \emph{left jump of the value~$a_j$ by $j-i$~steps} is a cyclic right rotation of this substring to $a_j a_i\cdots a_{j-1}$.

\subsection{The basic algorithm}

The following simple greedy algorithm was proposed in~\cite{MR4391718} to generate a set of permutations $L_n\seq S_n$.
We say that a jump is \emph{minimal} (w.r.t.~$L_n$), if every jump of the same value in the same direction by fewer steps creates a permutation that is not in~$L_n$.
Note that each entry of the permutation admits at most one minimal left jump and at most one minimal right jump.

\begin{algo}{Algorithm~J}{Greedy minimal jumps}
This algorithm attempts to greedily generate a set of permutations $L_n\seq S_n$ using minimal jumps starting from an initial permutation $\pi_0 \in L_n$.
\begin{enumerate}[label={\bfseries J\arabic*.}, leftmargin=8mm, noitemsep, topsep=3pt plus 3pt]
\item{} [Initialize] Visit the initial permutation~$\pi_0$.
\item{} [Jump] Generate an unvisited permutation from~$L_n$ by performing a minimal jump of the largest possible value in the most recently visited permutation.
If no such jump exists, or the jump direction is ambiguous, then terminate.
Otherwise visit this permutation and repeat~J2.
\end{enumerate}
\end{algo}

Put differently, in step~J2 we consider the entries $n,n-1,\ldots,2$ of the current permutation in decreasing order, and for each of them we check whether it allows a minimal left or right jump that creates a previously unvisited permutation, and we perform the first such jump we find, unless the same entry also allows a jump in the opposite direction, in which case we terminate.
If no minimal jump creates an unvisited permutation, we also terminate the algorithm.
For example, consider $L_4 = \{1243, 1423, 4123, 4213, 2134\}$.
Starting with $\pi_0 = 1243$, the algorithm generates~$\pi_1=1423$ (obtained from~$\pi_0$ by a left jump of the value~4 by 1~step), then~$\pi_2=4123$, then~$\pi_3=4213$ (in~$\pi_2$, 4 cannot jump, as~$\pi_0$ and~$\pi_1$ have been visited before; 3 cannot jump either to create any permutation from~$L_4$, so 2 jumps left by 1~step), then $\pi_4=2134$, successfully generating~$L_4$.
If instead we initialize with $\pi_0=4213$, then the algorithm generates $\pi_1=2134$, and then stops, as no further jump is possible.
If we choose $\pi_0=1423$, then we may jump~4 to the left or right (by 1~step), but as the direction is ambiguous, the algorithm stops immediately.
As mentioned before, the algorithm may stop before having visited the entire set~$L_n$ either because no minimal jump leading to a new permutation from~$L_n$ is possible, or because the direction of jump is ambiguous in some step.
By the definition of step~J2, the algorithm will never visit any permutation twice.

\subsection{Zigzag languages}

The following theorem, proved in~\cite{MR4391718}, provides a sufficient condition on the set~$L_n$ to guarantee that Algorithm~J is successful.
This condition is captured by the following closure property of the set~$L_n$.
A set of permutations~$L_n\seq S_n$ is called a \emph{zigzag language}, if either $n=0$ and $L_0=\{\varepsilon\}$, or if $n\geq 1$ and $L_{n-1}:=\{p(\pi)\mid \pi\in L_n\}$ is a zigzag language satisfying either one of the following conditions:
\begin{enumerate}[label={(z\arabic*)}, leftmargin=8mm, noitemsep, topsep=3pt plus 3pt]
\item For every $\pi\in L_{n-1}$ we have~$c_1(\pi)\in L_n$ and~$c_n(\pi)\in L_n$.
\item We have $L_n=\{c_n(\pi)\mid \pi\in L_{n-1}\}$.
\end{enumerate}

The definition of zigzag language given in~\cite{MR4391718} did not include condition~(z2), but only condition~(z1).
However, all results from our earlier paper carry over straightforwardly.
Essentially, condition~(z2) is a technicality we include here to be able to handle lattice congruences in full generality.
Condition~(z2) expresses that $L_n$ is obtained from~$L_{n-1}$ simply by inserting the new largest value~$n$ at the rightmost position of all permutations, i.e., the value~$n$ only ever appears to the right of~$1,\ldots,n-1$.
In this case we will have in particular $|L_n|=|L_{n-1}|$.

We now define a sequence~$J(L_n)$ of all permutations from a zigzag language~$L_n\seq S_n$.
For any $\pi\in L_{n-1}$ we let $\rvec{c}(\pi)$ be the sequence of all $c_i(\pi)\in L_n$ for $i=1,2,\ldots,n$, starting with~$c_1(\pi)$ and ending with~$c_n(\pi)$, and we let $\lvec{c}(\pi)$ denote the reverse sequence, i.e., it starts with~$c_n(\pi)$ and ends with~$c_1(\pi)$.
In words, those sequences are obtained by inserting into~$\pi$ the new largest value~$n$ from left to right, or from right to left, respectively, in all possible positions that yield a permutation from~$L_n$, skipping the positions that yield a permutation that is not in~$L_n$.
The sequence~$J(L_n)$ is defined recursively as follows:
If $n=0$ then we define $J(L_0):=\varepsilon$, and if~$n\geq 1$ then we consider the sequence $J(L_{n-1})=:\pi_1,\pi_2,\ldots$ and define
\begin{subequations}
\label{eq:JLn12}
\begin{equation}
\label{eq:JLn1}
  J(L_n)=\lvec{c}(\pi_1),\rvec{c}(\pi_2),\lvec{c}(\pi_3),\rvec{c}(\pi_4),\ldots
\end{equation}
if condition~(z1) holds, and we define
\begin{equation}
\label{eq:JLn2}
  J(L_n)=c_n(\pi_1),c_n(\pi_2),c_n(\pi_3),c_n(\pi_4),\ldots
\end{equation}
\end{subequations}
if condition~(z2) holds.

\begin{theorem}[\cite{MR4391718}]
\label{thm:jump}
Given any zigzag language of permutations~$L_n$ and initial permutation $\pi_0 = \ide_n$, Algorithm~J visits every permutation from~$L_n$ exactly once, in the order $J(L_n)$ defined by~\eqref{eq:JLn12}.
\end{theorem}

\section{Generating lattice congruences of the weak order}
\label{sec:cong}

In this section we show how Algorithm~J can be used to generate any lattice congruence of the weak order on~$S_n$.
The main results of this section are summarized in Theorem~\ref{thm:lattice} and Corollary~\ref{cor:quotient} below.

\subsection{Preliminaries}

We begin to recall a few basic definitions for a poset~$(P,<)$.
An \emph{antichain} in~$P$ is a set of pairwise incomparable elements.
A subset~$U\seq P$ is an \emph{upset} if~$x\in U$ and~$x<y$ implies that~$y\in U$.
Similarly, $D\seq P$ is a \emph{downset} if~$x\in D$ and~$y<x$ implies that~$y\in D$.
Clearly, the complement of an upset is a downset and vice versa.
Moreover, the minimal elements of an upset and the maximal elements of a downset form an antichain.
The \emph{upset of an element~$x\in P$} is the upset containing exactly all $y$ with~$x<y$.
Similarly, the \emph{downset of~$x$} is the downset containing exactly all $y$ with~$y<x$.
An \emph{interval}~$X=[x,y]$ in~$P$ is the intersection of the upset of~$x$ with the downset of~$y$, and we write $x=\min(X)$ and $y=\max(X)$.

A \emph{cover relation} is a pair $x,y\in P$ with~$x<y$ for which there is no~$z\in P$ with~$x<z<y$.
In this case we say that \emph{$y$ covers $x$} and we write~$x\lessdot y$.
We also refer to~$x$ as a \emph{down-neighbor} of~$y$, and to~$y$ as an \emph{up-neighbor} of~$x$.
Clearly, the cover relations form an acyclic directed graph with vertex set~$P$, and this graph is referred to as the \emph{cover graph of~$P$}, and its edges as \emph{cover edges}.
A drawing of the cover graph with all cover edges $x\lessdot y$ leading upwards is called a \emph{Hasse diagram}.
A poset~$(P,<)$ is called a \emph{lattice}, if for any two~$x,y\in P$ there is a unique smallest element~$z$, called the \emph{join~$x\vee y$ of~$x$ and~$y$}, such that~$z>x$ and~$z>y$, and if there is a unique largest element~$z$, called the \emph{meet~$x\wedge y$ of~$x$ and~$y$}, satisfying~$z<x$ and~$z<y$.
A \emph{lattice congruence} is an equivalence relation~$\equiv$ on~$P$ such that~$x\equiv x'$ and~$y\equiv y'$ implies that $x\vee y\equiv x'\vee y'$ and $x\wedge y\equiv x'\wedge y'$.
Given any lattice congruence~$\equiv$, we obtain the \emph{lattice quotient}~$P/{\equiv}$ (which is itself a lattice) by taking the equivalence classes as elements, and ordering them by~$X<Y$ if and only if there is an~$x\in X$ and a~$y\in Y$ such that~$x<y$ in~$P$.
Observe that the cover graph of~$P/{\equiv}$ is obtained from the cover graph of~$P$ by contracting all cover edges $x\lessdot y$ with~$x\equiv y$.
For any $x\in P$, we let $X_P(x)=X(x)$ denote the equivalence class in~$P/{\equiv}$ containing~$x$.

We will need the following two lemmas.

\begin{lemma}
\label{lem:interval}
For any lattice congruence of a finite lattice, every equivalence class is an interval.
\end{lemma}

\begin{lemma}
\label{lem:subposet}
Given a finite lattice~$(P,<)$ and any lattice congruence~$\equiv$, the lattice quotient $P/{\equiv}$ is isomorphic to the induced subposet of~$P$ whose elements are either the minima of the equivalence classes or the maxima.
\end{lemma}

Lemma~\ref{lem:interval} follows immediately from the definition of lattice congruence.
\iffull
Lemma~\ref{lem:subposet} has appeared in many previous papers, see e.g.~\cite{MR911564,MR1484432,MR1618652,MR1902662}.
\fi
\ifnotfull
Lemma~\ref{lem:subposet} is also well-known.
\fi
It can be proved by showing that given two equivalence classes~$X$ and~$Y$ of~$\equiv$ and two elements $x\in X$, $y\in Y$ with $x\lessdot y$, then we have $\min(X)<\min(Y)$ and $\max(X)<\max(Y)$.

Recall that the weak order on~$S_n$ is the order given by inclusion of inversion sets.
Note that the cover relations in this poset are exactly adjacent transpositions, i.e., swaps of two entries at neighboring positions in the permutation.
Observe also that the inversion set of the join~$\pi\vee\rho$ of two permutations~$\pi$ and~$\rho$ is given by the transitive closure of~$\inv(\pi)\cup \inv(\rho)$, and the inversion set of the meet can be computed similarly by considering the reverse permutations (which have the complementary inversion set).
In the weak order on~$S_n$, if two permutations~$\pi$ and~$\rho$ differ by transposing $a$ and~$b$, then we refer to the corresponding cover edge as an \emph{$(a,b)$-edge}, and if $\pi\equiv \rho$ then we refer to it as an \emph{$(a,b)$-bar}.
Bars are drawn with bold edges in all our figures.
The cover edges involving a fixed permutation $\pi=a_1\cdots a_n$ can be described more precisely by considering all \emph{ascents} of~$\pi$, i.e., all pairs~$(a_i,a_{i+1})$ with~$a_i<a_{i+1}$ and all \emph{descents} of~$\pi$, i.e., all pairs~$(a_i,a_{i+1})$ with $a_i>a_{i+1}$.
Specifically, for fixed~$\pi$, all cover edges $\pi\lessdot\rho$ are given by transposing the ascents of~$\pi$, and all cover edges $\pi\gtrdot\rho$ are given by transposing the descents of~$\pi$.
We let $\asc(\pi)$ and $\desc(\pi)$ denote the number of ascents and descents of~$\pi$, respectively.

\subsection{Combinatorics of lattice congruences of the weak order}
\label{sec:comb}

In the following discussion of lattice congruences of the weak order, we borrow some of the terminology and notation introduced by Reading~\cite{MR2037526,MR3335492}; see also his surveys~\cite{MR3221544, MR3645056, MR3645055}.

It is clear from the definition of lattice congruence, that if certain permutations are equivalent, this also forces other permutations to be equivalent.
These relations on the cover edges are expressed by \emph{forcing constraints}.
The two forcing constraints that are relevant for us are shown in Figure~\ref{fig:forcing}.
We refer to them as type~i and type~ii constraints, shown on the left and right of the figure, respectively.
A type~i constraint involves four permutations $\pi,\rho,\pi',\rho'$ satisfying $\pi\lessdot\rho\lessdot\rho'$ and $\pi\lessdot\pi'\lessdot\rho'$ that differ in adjacent transpositions of two values~$a,b$ or two values~$c,d$ with~$a<b$ and~$c<d$, as shown in the figure.
This constraint expresses that~$\pi\equiv\rho$ if and only if~$\pi'\equiv \rho'$, i.e., either both $(a,b)$-edges $(\pi,\rho)$ and~$(\pi',\rho')$ are bars or both are non-bars.
A type~ii constraint involves six permutations $\pi,\rho,\pi',\rho',\sigma,\tau$ satisfying $\pi\lessdot\rho\lessdot\tau\lessdot\rho'$ and $\pi\lessdot\sigma\lessdot\pi'\lessdot\rho'$ that differ in three adjacent values~$a,b,c$ with~$a<b<c$, as shown in the figure.
This constraint expresses that~$\pi\equiv\rho$ if and only if~$\pi'\equiv \rho'$, and moreover these conditions imply~$\sigma\equiv \pi'$ and~$\tau\equiv\rho$ (but not the converse), i.e., the first two $(a,b)$-edges are both either bars or non-bars, and in the first case they also force the latter two $(a,c)$-edges to be bars.
Note that both constraints follow immediately from the definition of lattice congruence, and that they are meant to capture also the symmetric situation obtained by reversing all permutations involved in Figure~\ref{fig:forcing}.

\begin{figure}
\includegraphics{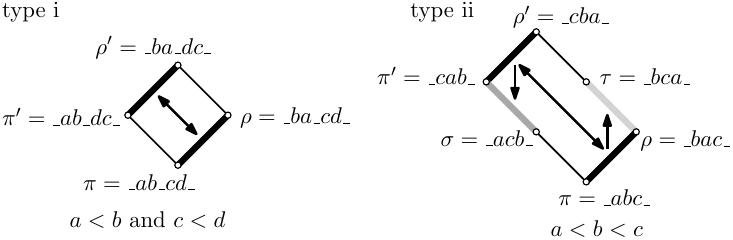}
\caption{Forcing constraints in a lattice congruence of the weak order.
Bold edges indicate bars, i.e., pairs of permutations that differ in an adjacent transposition and that belong to the same equivalence class.}
\label{fig:forcing}
\end{figure}

We now consider maximal sets of cover edges that are either all bars or all non-bars in any lattice congruence.
Given an $(a,b)$-bar, then type~i constraints allow us to reorder the values to the left or right of~$a$ and~$b$ in the corresponding permutations arbitrarily.
Moreover, given an $(a,b)$-bar, then type~ii constraints allow us to move any value that is larger or smaller than~$a$ and~$b$ to the left or right of them.
Consequently, a maximal set of mutually forcing bars is characterized by the pair~$(a,b)$, and by the values that are strictly between~$a$ and~$b$ and to the left of them.
This motivates the following definition:
Given a triple~$(a,b,L)$ with $1\leq a<b\leq n$ and~$L\seq \left]a,b\right[$, the \emph{fence $f(a,b,L)$} is the set of all $(a,b)$-edges, where the values in~$L$ are to the left of~$a$ and~$b$ in the corresponding permutations, the values in~$\left]a,b\right[\setminus L$ are to the right of~$a$ and~$b$, and the position of the remaining values $[n]\setminus [a,b]$ is arbitrary.
Note that the edges of any fence form a matching in the cover graph.
For instance, for~$n=4$ the fence~$f(2,4,\{3\})$ contains the $(2,4)$-edges $(3241,3421)$, $(1324,1342)$, and~$(3124,3142)$ that are mutually forcing bars; see Figure~\ref{fig:fences}.
In the figure, we visualize fences by an \emph{arc diagram}, which consists of a vertical sequence of $n$~points labeled $1,\ldots,n$ from bottom to top, and for every fence~$f(a,b,L)$ there is an arc joining the $a$th and $b$th point, with the points in~$L$ left of the arc, and the points in $\left]a,b\right[\setminus L$ right of the arc.
We let
\begin{equation*}
F_n:=\big\{f(a,b,L)\mid 1\leq a<b\leq n\text{ and } L\seq \left]a,b\right[\big\}
\end{equation*}
denote the set of all fences.

\begin{figure}
\includegraphics{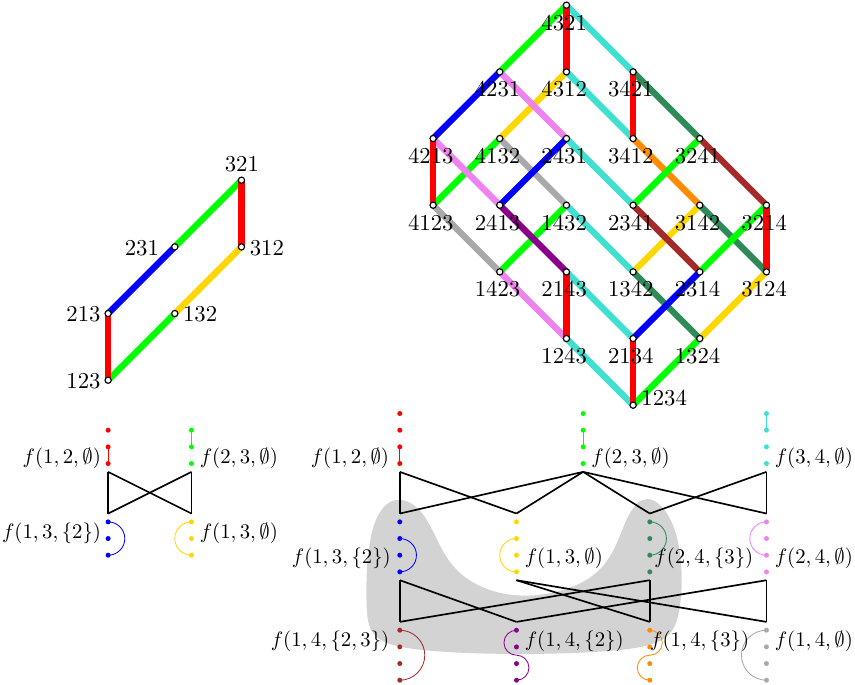}
\caption{Illustration of fences and the forcing order for~$n=3$ (left) and~$n=4$ (right).
Cover edges of the same fence are drawn in the same color.
The highlighted region shows a downset in the forcing order, corresponding to the lattice quotient in Figure~\ref{fig:cong}.
}
\label{fig:fences}
\end{figure}

The (non-mutual) forcing constraints between fences induced by type~ii constraints yield a partial order on~$F_n$, called the \emph{forcing order}.
Specifically, two fences~$f(a,b,L)$ and~$f(c,d,M)$ satisfy $f(a,b,L)\prec f(c,d,M)$ in the forcing order, if $a\leq c<d\leq b$, $(a,b) \neq (c,d)$, and $M=L\cap \left]c,d\right[$.
Note that two such fences form a cover relation in the forcing order if and only if $(c,d)=(a+1,b)$ or $(c,d)=(a,b-1)$.
Consequently, every non-maximal fence~$f(a,b,L)$ is covered by two other fences, and every non-minimal fence~$f(a,b,L)$ covers two fences if either~$a=0$ or~$b=n$, and four fences if~$0<a<b<n$.
The interpretation is that if $f(a,b,L)\prec f(c,d,M)$, then the bars of the fence~$f(c,d,M)$ force the bars of the fence~$f(a,b,L)$, i.e., forcing goes downward in the forcing order.
For example, we have $f(1,4,\{2,3\})\prec f(2,4,\{3\})$, i.e., the three bars $(3241,3421)$, $(1324,1342)$, and~$(3124,3142)$ from before force the two bars~$(2314,2341)$ and~$(3214,3241)$.

\begin{theorem}[{\cite[Section~10-5]{MR3645056}}]
\label{thm:reading}
For every lattice congruence~$\equiv$ of the weak order on~$S_n$, there is a subset of fences $F_\equiv\seq F_n$ such that in each equivalence class of~$\equiv$, all cover edges are a bar from a fence in~$F_\equiv$, and all other cover edges are not in any fence from~$F_\equiv$.
Moreover, $F_\equiv$ is a downset of the forcing order~$\prec$ and the map $\equiv{}\mapsto F_\equiv$ is a bijection between the lattice congruences of the weak order on~$S_n$ and the downsets of the forcing order~$\prec$.
\end{theorem}

From now on we use $F_\equiv$ as the set of fences corresponding to a lattice congruence~$\equiv$ given by Theorem~\ref{thm:reading}.
The downset $F_\equiv$ describes exactly all the cover edges that are contracted to obtain the lattice quotient~$S_n/{\equiv}$.
Equivalently, the upset $F_n\setminus F_\equiv$ describes all cover edges that are not contracted in the quotient.

In the dual setting of hyperplane arrangements considered in~\cite{MR2037526,MR3964495}, the dual of a fence is called a shard.
Moreover, these authors represent a lattice congruence~$\equiv$ not by the set of fences~$F_\equiv$ that contains all cover edges that are contracted to obtain the lattice quotient~$S_n/{\equiv}$, but by the set $F_n\setminus F_\equiv$ of cover edges that are not contracted in the quotient.
The latter representation allows describing each equivalence class by a non-crossing arc diagram that contains only arcs corresponding to fences from~$F_n\setminus F_\equiv$~\cite{MR3335492}.
On the other hand, our representation makes the characterization of congruences with regular and vertex-transitive quotient graphs in Section~\ref{sec:reg} somewhat more natural.

We may order all downsets of the forcing order by inclusion, yielding another lattice; see Figure~\ref{fig:arcs}.
By Theorem~\ref{thm:reading}, this corresponds to ordering all lattice congruences of the weak order on~$S_n$ by refinement.
The finest lattice congruence~$\equiv$ does not use any fences $F_\equiv=\emptyset$, and corresponds to the set of all permutations~$S_n$, and the coarsest lattice congruence~$\equiv$ uses all fences~$F_\equiv=F_n$, and corresponds to contracting all permutations into a single equivalence class.

\begin{figure}
\includegraphics[width=\textwidth]{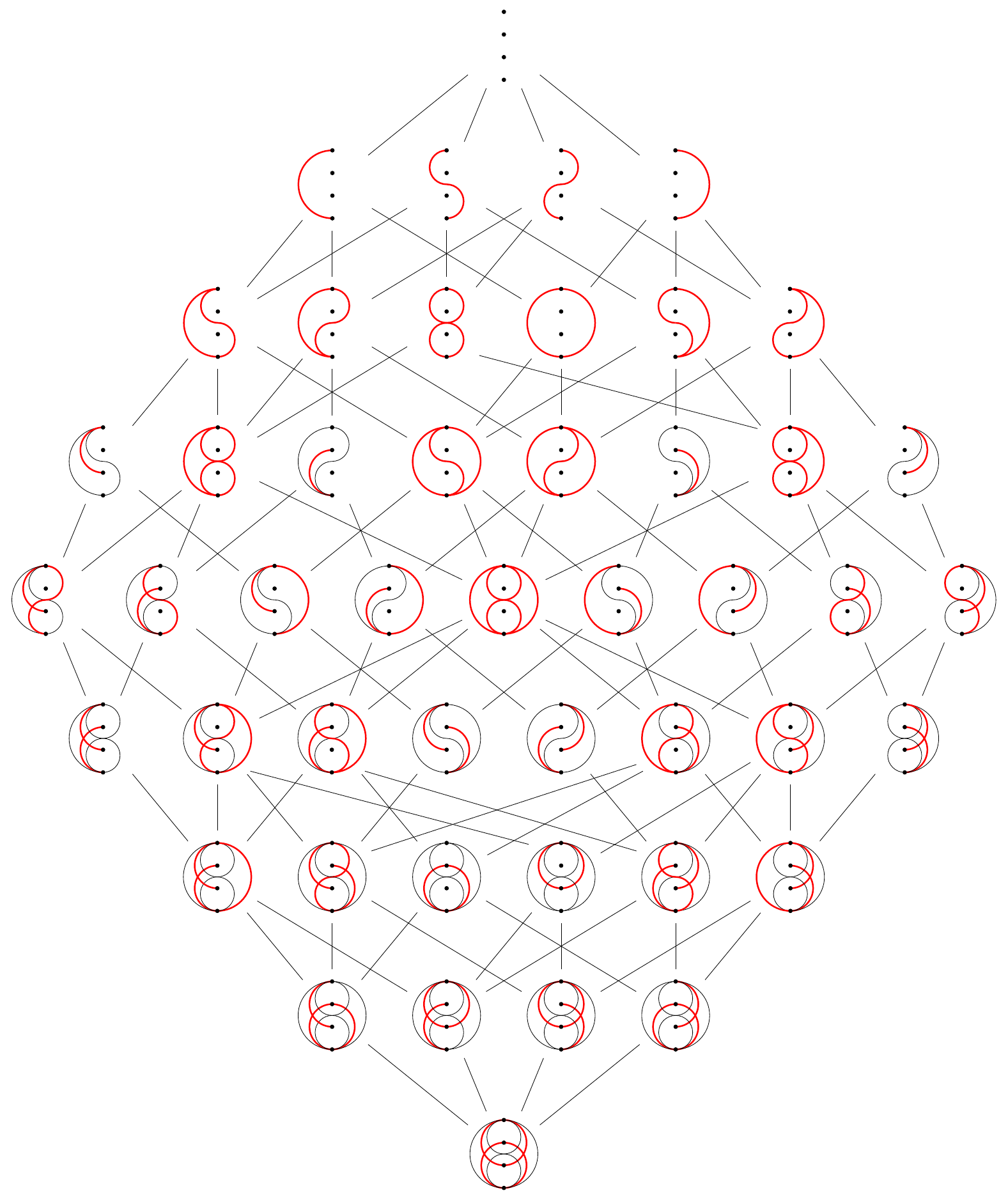}
\caption{Lattice of congruences of the weak order on~$S_n$ for~$n=4$, represented by downsets of the forcing order.
Each downset of fences is represented by an arc diagram containing all the corresponding arcs, where the arcs corresponding to maximal fences of the downset are highlighted.
The figure shows only downsets not containing any non-essential fences~$f(a,a+1,\emptyset)$, $a\in[n-1]$, as otherwise the congruence is equivalent to a lower-dimensional one (see Lemma~\ref{lem:dim} below).
}
\label{fig:arcs}
\end{figure}

\subsection{Restrictions, rails, ladders, and projections}

Given a lattice congruence~$\equiv$ of the weak order on~$S_n$, the \emph{restriction of~$\equiv$}, denoted $\equiv^*$, is the relation on~$S_{n-1}$ induced by all permutations that have the largest value~$n$ at the last position, i.e., it is the set of all pairs~$(\pi,\rho)$ with $\pi,\rho\in S_{n-1}$ for which $c_n(\pi)\equiv c_n(\rho)$.

\begin{lemma}
\label{lem:restrict}
For every lattice congruence~$\equiv$ of the weak order on~$S_n$, the restriction~$\equiv^*$ is a lattice congruence on~$S_{n-1}$.
\end{lemma}

\begin{proof}
Clearly, for any two permutations~$\pi,\rho\in S_{n-1}$ we have
\begin{equation}
\label{eq:commute}
c_n(\pi)\vee c_n(\rho)=c_n(\pi\vee\rho) \quad \text{and} \quad c_n(\pi)\wedge c_n(\rho)=c_n(\pi\wedge\rho).
\end{equation}
Now consider four permutations $\pi, \pi', \rho, \rho'\in S_{n-1}$ satisfying $\pi \equiv^* \pi'$ and $\rho \equiv^* \rho'$.
From the definition of restriction, we have $c_n(\pi) \equiv c_n(\pi')$ and $c_n(\rho) \equiv c_n(\rho')$.
Applying the definition of lattice congruence to~$\equiv$, we obtain that $c_n(\pi) \vee c_n(\rho) \equiv c_n(\pi') \vee c_n(\rho')$ and $c_n(\pi) \wedge c_n(\rho) \equiv c_n(\pi') \wedge c_n(\rho')$.
Applying~\eqref{eq:commute} to these relations yields $c_n(\pi \vee \rho) \equiv c_n(\pi' \vee \rho')$ and $c_n(\pi \wedge \rho) \equiv c_n(\pi' \wedge \rho')$, from which we obtain $\pi \vee \rho \equiv^* \pi' \vee \rho'$ and $\pi \wedge \rho \equiv^* \pi' \wedge \rho'$ with the definition of restriction.
This proves the lemma.
\end{proof}

\begin{figure}
\includegraphics{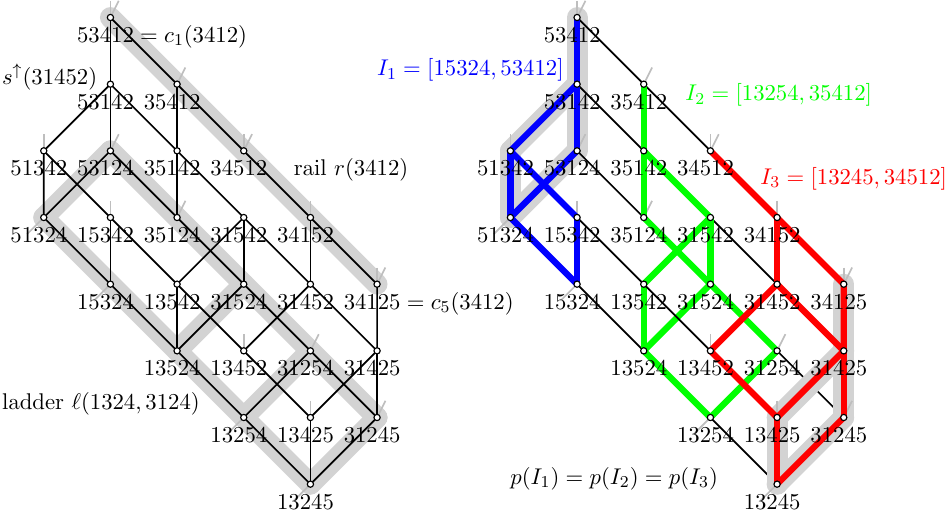}
\caption{Illustration of rails, ladders, and projections.
The figure shows only a subset of permutations from~$S_5$.
The equivalence classes on the right are shown as intervals.
}
\label{fig:ladder}
\end{figure}

The following definitions are illustrated in Figure~\ref{fig:ladder}.
Recall that for any permutation~$\pi\in S_{n-1}$ and for any~$1\leq i\leq n$, the permutation~$c_i(n)$ is obtained from~$\pi$ by inserting the largest value~$n$ at position~$i$.
Given any permutation~$\pi\in S_{n-1}$, we refer to the cover edges $c_n(\pi)\lessdot c_{n-1}(\pi)\lessdot\cdots\lessdot c_1(\pi)$ in~$S_n$ as the \emph{rail}~$r(\pi)$.
Given two permutations $\pi,\rho\in S_{n-1}$ with~$\pi\lessdot\rho$, we refer to the cover edges of the weak order induced by the permutations on the rails of~$\pi$ and~$\rho$ as the \emph{ladder}~$\ell(\pi,\rho)$.
Let~$k$ and~$k+1$ be the positions in which~$\pi$ and~$\rho$ differ.
Note that the ladder~$\ell(\pi,\rho)$ has exactly all cover edges of the rails, plus the cover edges $c_i(\pi)\lessdot c_i(\rho)$ for all~$1\leq i\leq n$ except for~$i=k+1$, which are referred to as the \emph{stairs} of the ladder.
We see that the cover graph of the weak order on~$S_n$ has the following recursive structure: It is the union of all ladders~$\ell(\pi,\rho)$ obtained from all cover edges~$\pi\lessdot\rho$ with~$\pi,\rho\in S_{n-1}$.

\begin{lemma}
\label{lem:collapse}
For every lattice congruence~$\equiv$ of the weak order on~$S_n$, the following three statements are equivalent:
\begin{enumerate}[label=(\roman*), leftmargin=8mm, noitemsep, topsep=3pt plus 3pt]
\item $\ide_n\equiv c_{n-1}(\ide_{n-1})$, i.e., the identity permutation and the one obtained from it by transposing the last two entries form a bar.
\item There is a permutation~$\pi\in S_{n-1}$ such that for all $1\leq i<n$ we have $c_i(\pi)\equiv c_{i+1}(\pi)$, i.e., the rail~$r(\pi)$ consists entirely of bars.
\item For all permutations~$\pi\in S_{n-1}$ and all $1\leq i<n$ we have $c_i(\pi)\equiv c_{i+1}(\pi)$, i.e., all rails~$r(\pi)$ consist entirely of bars.
\end{enumerate}
\end{lemma}

\begin{proof}
Clearly, (iii) implies~(ii) and (iii) implies~(i), so it suffices to prove that (ii) implies~(iii) and that (i) implies~(iii).
We prove this by showing that if there is an $(n-1,n)$-bar in~$S_n$, then (iii) follows.
If there is an $(n-1,n)$-bar, this means that the fence $f(n-1,n,\emptyset)$ is in~$F_\equiv$.
However, as $F_\equiv$ is a downset in the forcing order (recall Theorem~\ref{thm:reading}), it follows that all fences~$f(a,n,L)$ with $1\leq a\leq n-1$ and an arbitrary subset $L\seq \left]a,n\right[$ are also in~$F_\equiv$.
For any~$\pi=a_1\cdots a_{n-1}\in S_{n-1}$, the $i$th edge along the rail $r(\pi)=c_n(\pi)\lessdot c_{n-1}(\pi)\lessdot\cdots\lessdot c_1(\pi)$ is an $(a_{n-i},n)$-edge, so it is a bar regardless of the values of $a_1\cdots a_{n-i}$.
\end{proof}

Combining Lemmas~\ref{lem:interval} and~\ref{lem:collapse} yields the following lemma.

\begin{lemma}
\label{lem:rail}
Let $\equiv$ be an equivalence relation of the weak order on~$S_n$ with $\ide_n\not\equiv c_{n-1}(\ide_{n-1})$.
Then for every rail~$r(\pi)$, $\pi\in S_{n-1}$, and every equivalence class $X\in S_n/{\equiv}$ we have that $X\cap r(\pi)$ is an interval of~$r(\pi)$.
Moreover, there are two distinct equivalence classes~$X$ and~$Y$ containing the first and last permutation of the rail, i.e., $c_n(\pi)\in X$ and~$c_1(\pi)\in Y$.
\end{lemma}

Recall that for~$\pi\in S_n$, the permutation $p(\pi)\in S_{n-1}$ is obtained by removing the largest value~$n$ from~$\pi$.
Given a set of permutations $X\seq S_n$, we refer to $p(X):=\{p(\pi)\mid \pi\in X\}$ as the \emph{projection of~$X$}.
This definition and the following crucial lemma are illustrated in Figure~\ref{fig:ladder}.

\begin{lemma}
\label{lem:projection}
For every lattice congruence~$\equiv$ of the weak order on~$S_n$ and every equivalence class~$X$ of~$\equiv$, we have that the projection~$p(X)$ is an equivalence class of the restriction~$\equiv^*$.
In particular, any two equivalence classes~$X,Y$ of~$S_n/{\equiv}$ either have the same projection $p(X)=p(Y)$ or disjoint projections $p(X)\cap p(Y)=\emptyset$.
\end{lemma}

The proof of this lemma essentially proceeds by repeatedly applying the forcing constraints shown in Figure~\ref{fig:forcing} along ladders.
However, we do not apply these constraints directly, but using the fences captured by Theorem~\ref{thm:reading}.

\begin{proof}
For any~$n\geq 1$ and any permutation $\pi\in S_n$, we let~$N(\pi)$ denote the set of all permutations that differ from~$\pi$ in an adjacent transposition, i.e., all neighbors in the cover graph of~$S_n$.
Now consider a fixed lattice congruence~$\equiv$ on~$S_n$, fix an equivalence class~$X$ of~$\equiv$ and some permutation~$\pi\in X$, and consider its projection $\pi':=p(\pi)\in S_{n-1}$.
The lemma is a consequence of the following two statements:
\begin{enumerate}[label=(\roman*), leftmargin=8mm, noitemsep, topsep=3pt plus 3pt]
\item For every $\rho\in N(\pi)\seq S_n$ with~$\pi\equiv \rho$ we have that~$p(\pi)\equiv^* p(\rho)$.
\item For every $\rho'\in N(\pi')\seq S_{n-1}$ with~$\pi'\equiv^* \rho'$ there is a~$\rho\in N(\pi)$ with~$\pi\equiv \rho$ and~$p(\rho)=\rho'$, or there is a~$\sigma\in N(\pi)$ and a~$\rho\in N(\sigma)$ with~$\pi\equiv\sigma\equiv \rho$ and~$p(\pi)=p(\sigma)=\pi'$ and~$p(\rho)=\rho'$.
\end{enumerate}
In words, (i) asserts that the projection of any bar incident to~$\pi$ is a bar incident to~$\pi'$ in the restriction, and (ii) asserts that for any bar incident to~$\pi'$ in the restriction, there are one or two consecutive bars starting at~$\pi$ whose projection is this bar.

We begin proving~(i).
Let $\rho\in N(\pi)\seq S_n$ with~$\pi\equiv \rho$.
If~$\pi$ and~$\rho$ are endpoints of an $(a,n)$-bar for some~$a<n$ (i.e., this bar is part of the rail~$r(\pi')$), then we have~$p(\pi)=p(\rho)$, so trivially~$p(\pi)\equiv^* p(\rho)$.
Otherwise $\pi$ and~$\rho$ are endpoints of some $(a,b)$-bar for $a<b<n$ (i.e., this bar is a stair of some ladder), so the fence~$f(a,b,L)$ is in~$F_\equiv$, where $L$ is the set of all values from~$\left]a,b\right[$ left of~$a$ and~$b$ in~$\pi$ and~$\rho$.
By the definition of a fence, it follows that $c_n(p(\pi))\equiv c_n(p(\rho))$, i.e., the permutations obtained from~$\pi$ and~$\rho$ by moving the largest value~$n$ to the rightmost position are equivalent.
By the definition of restriction, we obtain that~$p(\pi)\equiv^* p(\rho)$, as claimed.

We now prove~(ii).
Let $\rho'\in N(\pi')\seq S_{n-1}$ with~$\pi'\equiv^* \rho'$.
Clearly, $\pi'$ and~$\rho'$ are endpoints of some $(a,b)$-bar in~$\equiv^*$ for $a<b\leq n-1$.
By the definition of restriction, it follows that $c_n(\pi')\equiv c_n(\rho')$, so $f(a,b,L)$ is a fence in~$F_\equiv$, where $L$ is the set of all values from~$\left]a,b\right[$ left of~$a$ and~$b$ in~$\pi'$ and~$\rho'$.
In the following we assume that~$\pi'\lessdot\rho'$, i.e., $\pi'$ contains the ascent~$(a,b)$, and $\rho'$ contains the descent~$(b,a)$.
Let $i$ be such that~$\pi=c_i(\pi')$, and let~$k$ be the position of~$b$ in~$\pi'$.
We now distinguish two cases.
If~$i\neq k$, then $\pi=c_i(\pi')\lessdot c_i(\rho')$ is a cover edge in~$S_n$ (it is a stair of the ladder~$\ell(\pi',\rho')$), and since it is contained in the fence~$f(a,b,L)$, we have $\pi=c_i(\pi')\equiv c_i(\rho')$, i.e., this cover edge is indeed a bar.
This means we can take $\rho:=c_i(\rho')\in N(\pi)$, which satisfies~$p(\rho)=\rho'$ by definition.
On the other hand, if~$i=k$, then~$c_i(\pi')$ and~$c_i(\rho')$ are \emph{not} endpoints of a cover edge (this is the missing stair in the ladder~$\ell(\pi',\rho')$).
However, we may take $\sigma:=c_{i-1}(\pi')\in N(\pi)$ and $\rho:=c_{i-1}(\rho')\in N(\sigma)$ (note that $i=k\geq 2$), and then~$\pi\lessdot\sigma$ is an $(a,n)$-edge, and $\sigma\lessdot\rho$ is an $(a,b)$-edge.
As $f(a,b,L)$ is a fence in~$F_\equiv$, the forcing order implies that~$f(a,n,L')$ is also a fence, where~$L'$ is defined as the set of all values from~$\left]a,n\right[$ left of~$a$ and~$n$ in~$\pi$ and~$\sigma$.
Consequently, we have $\pi\equiv\sigma\equiv \rho$, and moreover $p(\pi)=p(\sigma)=\pi'$ and~$p(\rho)=\rho'$ by the definition of~$\sigma$ and~$\rho$, i.e., these two cover edges are indeed bars.
In the remaining subcase~$\pi'\gtrdot\rho'$ we can take $\rho:=c_i(\rho')\in N(\pi)$ if~$i\neq k$, and $\sigma:=c_{i+1}(\pi')$ and $\rho:=c_{i+1}(\rho')$ if~$i=k$, and argue similarly to before.

This proves the lemma.
\end{proof}

We state the following two lemmas for further reference.
The first lemma is an immediate consequence of Lemma~\ref{lem:projection}.
For any lattice congruence~$\equiv$ of the weak order on~$S_n$ and any fence~$f(a,b,L)$ in~$F_\equiv$ with $b<n$, we let $f^*(a,b,L)$ denote the fence formed by the union of all $(a,b)$-edges in the weak order on~$S_{n-1}$ obtained by removing the largest value~$n$ from all permutations of~$f(a,b,L)$.

\begin{lemma}
\label{lem:restrict-fences}
For every lattice congruence~$\equiv$ of the weak order on~$S_n$, its restriction $\equiv^*$ satisfies $F_{\equiv^*}=\{f^*(a,b,L)\mid f(a,b,L)\in F_\equiv \text{ and } b<n\}$.
\end{lemma}

Rephrased in terms of arc diagrams, Lemma~\ref{lem:restrict-fences} asserts that the arc diagram of the restriction~$\equiv^*$ is obtained from the arc diagram of~$\equiv$ simply by removing the highest point labeled~$n$, and by deleting all arcs incident to it.

\begin{lemma}
\label{lem:restrict-minmax}
For every lattice congruence~$\equiv$ of the weak order on~$S_n$ and any equivalence class~$X\in S_n/{\equiv}$, consider its minimum $\pi:=\min(X)$ and maximum $\rho:=\max(X)$.
Then their projections~$p(\pi)$ and~$p(\rho)$ are the minimum and maximum of the equivalence class~$p(X)$ of the restriction~$\equiv^*$.
\end{lemma}

\begin{proof}
Suppose for the sake of contradiction that the maximum of~$p(X)$ is not~$p(\rho)$, but another permutation~$\sigma\in S_{n-1}$.
As $\sigma\in p(X)$, we obtain from Lemma~\ref{lem:projection} that $c_i(\sigma) \in X$ for some $1\leq i\leq n$.
As $\sigma$ is the unique maximum of~$p(X)$ (recall Lemma~\ref{lem:interval}), there exist two entries $a,b$ with $a<b$ that are inverted in~$\sigma$, i.e., $b$ appears before $a$ in~$\sigma$, but not in $p(\rho)$.
As inserting~$n$ into a permutation does not change the relative order of~$a$ and~$b$, the entries $a,b$ are also inverted in $c_i(\sigma)$, but not in~$\rho$.
However, by the definition of the weak order on~$S_n$, this means that $c_i(\sigma)\not<\rho$, contradicting the fact that $\rho$ is the maximum of~$X$.
A similar argument shows that $p(\pi)$ is the minimum of~$p(X)$.
\end{proof}

\subsection{Jumping through lattice congruences}

For any lattice congruence~$\equiv$ of the weak order on~$S_n$, a \emph{set of representatives} for the equivalence classes~$S_n/{\equiv}$ is a subset~$R_n\seq S_n$ such that for every equivalence class~$X\in S_n/{\equiv}$, exactly one permutation is contained in~$R_n$, i.e., $|X\cap R_n|=1$.
Recall that $X(\pi)$, $\pi\in S_n$, denotes the equivalence class from~$S_n/{\equiv}$ containing~$\pi$.
A meaningful definition of `generating the lattice congruence' is to generate a set of representatives for its equivalence classes.
We also require that any two successive representatives form a cover relation in the lattice quotient~$S_n/{\equiv}$.
This is what we achieve with the help of Algorithm~J.

We recursively define such a set of representatives~$R_n$ as follows; see Figure~\ref{fig:jump}:
If~$n=0$ then $R_0:=\{\varepsilon\}$, and if~$n\geq 1$ then we first compute the representatives~$R_{n-1}$ for the restriction~$\equiv^*$ to~$S_{n-1}$, and we then distinguish two cases:
If $\ide_n\not\equiv c_{n-1}(\ide_{n-1})$, then we consider every representative~$\pi\in R_{n-1}$, the corresponding rail~$r(\pi)$ in~$S_n$, and from every equivalence class~$X\in S_n/{\equiv}$ with $X\cap r(\pi)\neq\emptyset$ we pick exactly one permutation from~$X\cap r(\pi)$.
In particular, we always pick~$c_1(\pi)$ and~$c_n(\pi)$, which is possible by~Lemma~\ref{lem:rail}, yielding a set~$R_\pi$.
We then take the union of those permutations,
\begin{subequations}
\label{eq:Rn}
\begin{equation}
\label{eq:Rn1}
  R_n:=\bigcup_{\pi\in R_{n-1}} R_\pi.
\end{equation}
On the other hand, if $\ide_n\equiv c_{n-1}(\ide_{n-1})$ we define
\begin{equation}
\label{eq:Rn2}
  R_n:=\{c_n(\pi)\mid \pi\in R_{n-1}\}.
\end{equation}
\end{subequations}

\begin{figure}
\includegraphics{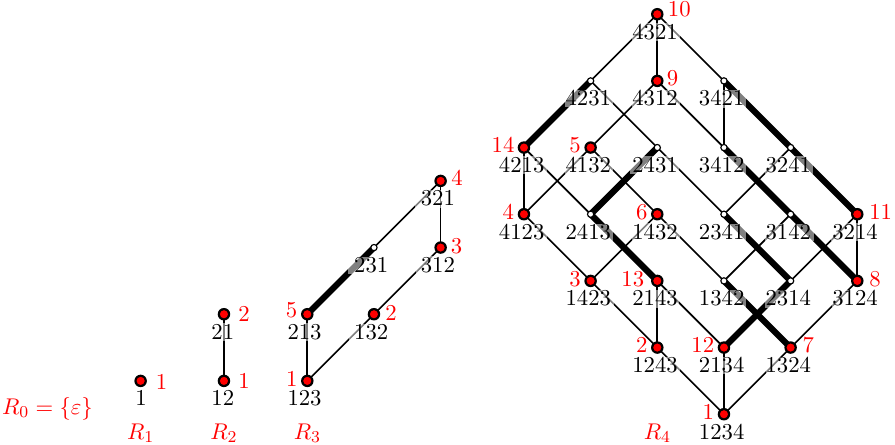}
\caption{Illustration of the representatives and the jumping order for the lattice congruence shown in Figure~\ref{fig:cong}.
The filled dots are the permutations in the sets~$R_n$, and the small numbers next them indicate the ordering in the sequences~$J(R_n)$ defined in~\eqref{eq:JLn12}.}
\label{fig:jump}
\end{figure}

\begin{lemma}
\label{lem:cong}
For every lattice congruence~$\equiv$ of the weak order on~$S_n$, the set~$R_n\seq S_n$ defined in~\eqref{eq:Rn} is indeed a set of representatives for~$S_n/{\equiv}$.
Moreover, $R_n$ is a zigzag language satisfying condition~(z1) if \eqref{eq:Rn1} holds, and condition~(z2) if \eqref{eq:Rn2} holds.
\end{lemma}

\begin{proof}
We argue by induction on~$n$.
The statement clearly holds for~$n=0$.
For the induction step, suppose that~$R_{n-1}$ is a set of representatives for the equivalence classes of~$S_{n-1}/{\equiv}^*$, and that $R_{n-1}$ is a zigzag language.
If $\ide_n\not\equiv c_{n-1}(\ide_{n-1})$, we obtain from Lemma~\ref{lem:projection} that for every equivalence class~$X$ of~$S_n/{\equiv}$, the projection~$p(X)$ is an equivalence class of the restriction~$\equiv^*$.
Therefore, we know by induction that~$R_{n-1}$ contains a unique representative~$\pi\in S_{n-1}$ for~$p(X)$, so by our choice of~$R_\pi$ we indeed have~$|X\cap R_\pi|=1$, and moreover $R_n$ as defined in~\eqref{eq:Rn1} satisfies~$|X\cap R_n|=1$.
Furthermore, as we chose~$R_\pi$ to contain~$c_1(\pi)$ and~$c_n(\pi)$ for all~$\pi\in S_{n-1}$, we obtain that~$R_n$ is a zigzag language satisfying condition~(z1) in the definition.
On the other hand, if $\ide_n\equiv c_{n-1}(\ide_{n-1})$, then we obtain from Lemma~\ref{lem:collapse} and Lemma~\ref{lem:projection} that every equivalence class~$X$ of~$S_n/{\equiv}$ satisfies $X=\{c_1(\pi),\ldots,c_n(\pi)\mid \pi\in p(X)\}$, showing that $R_n$ as defined in~\eqref{eq:Rn2} is indeed a set of representatives for~$S_n/{\equiv}$.
Moreover, in this case $R_n$ is a zigzag language satisfying condition~(z2) in the definition.
This completes the proof.
\end{proof}

\begin{figure}
\makebox[0cm]{ 
\includegraphics[width=18cm]{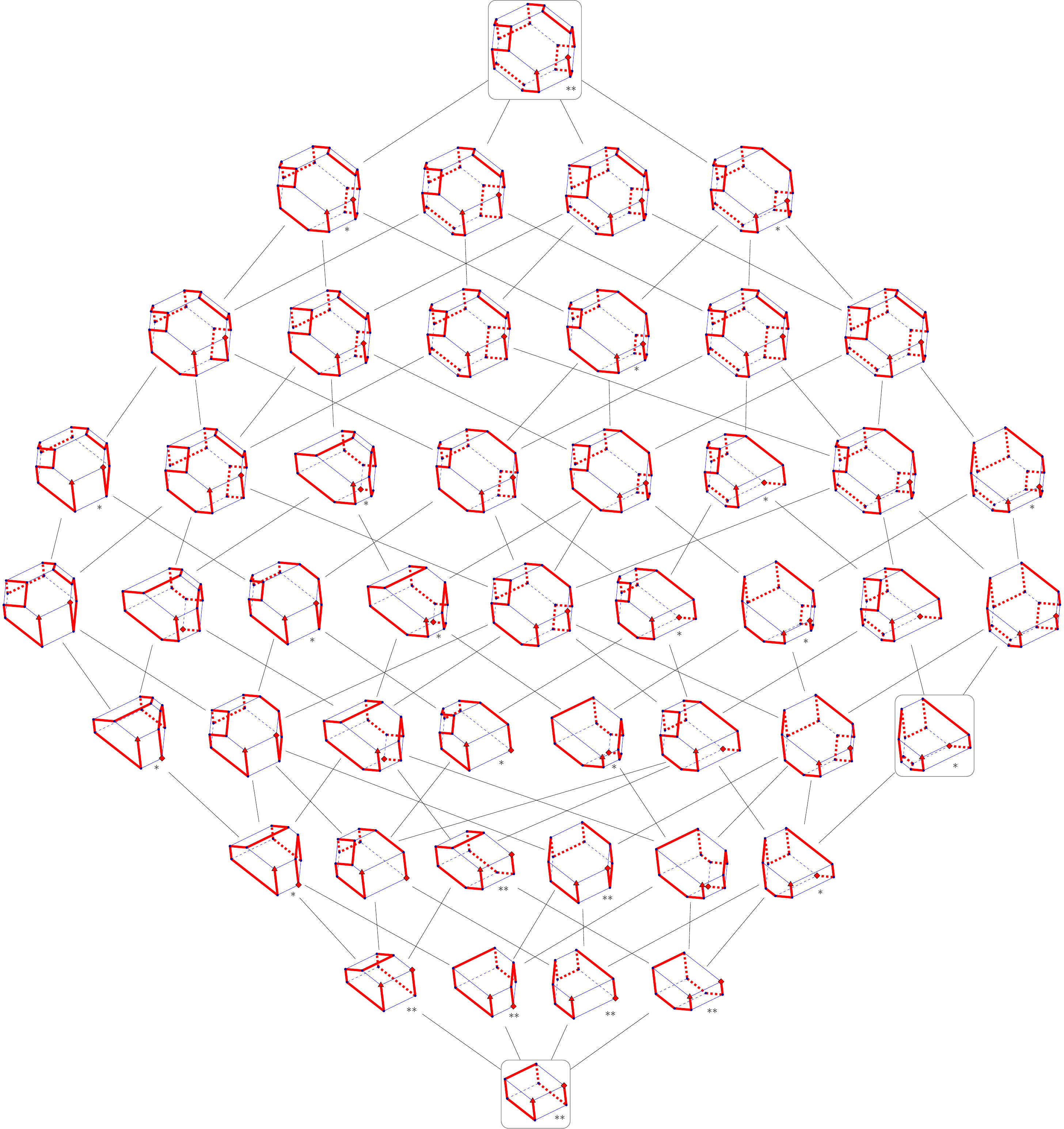}
}
\caption{Lattice congruences of the weak order on~$S_4$, ordered by refinement and realized as polytopes, where only the full-dimensional polytopes are shown.
The polytopes are arranged in the same way as in Figure~\ref{fig:arcs}.
The figure shows the Hamilton path on each quotientope computed by Algorithm~J, with the start and end vertex indicated by a triangle and diamond, respectively.
Permutahedron (top), associahedron (one of four isomorphic variants; middle right) and 3-cube (bottom) are highlighted.
The graphs marked with~* are regular, and those marked with~** are vertex-transitive.
}
\label{fig:quotient}
\end{figure}

\begin{lemma}
\label{lem:jump-order}
Running Algorithm~J with input $L_n:=R_n$, where $R_n$ is the set of representatives of a lattice congruence~$\equiv$ defined in~\eqref{eq:Rn}, then for any two permutations $\pi,\rho\in R_n$ that are visited consecutively, $X(\pi)$ and~$X(\rho)$ form a cover relation in the quotient~$S_n/{\equiv}$.
\end{lemma}

\begin{proof}
Let $R_n$ be a set of representatives of a lattice congruence~$\equiv$ defined in~\eqref{eq:Rn}, and consider the set $L_n:=R_n$, which is a zigzag language by Lemma~\ref{lem:cong}.
If \eqref{eq:Rn1} holds, then by Lemma~\ref{lem:cong} the set~$R_n$ satisfies condition~(z1), so the permutations of~$L_n=R_n$ are generated in the sequence~$J(L_n)$ defined in~\eqref{eq:JLn1}.
Observe that all permutations in~$\lvec{c}(\pi_k)$ or~$\rvec{c}(\pi_k)$, $\pi_k\in R_{n-1}\seq S_{n-1}$, lie on the rail~$r(\pi_k)$.
If $\pi,\rho\in R_n$ are visited consecutively and lie on the same rail, i.e., $\pi=c_i(\pi_k)$ and~$\rho=c_j(\pi_k)$ with $1\leq i<j\leq n$, then there is an integer~$s$ with $i\leq s<j$ such that
\begin{equation*}
  \pi=c_i(\pi_k)\equiv c_{i+1}(\pi_k)\equiv \cdots \equiv c_s(\pi_k)\not\equiv c_{s+1}(\pi_k)\equiv c_{s+2}(\pi_k)\equiv\cdots\equiv c_j(\pi_k)=\rho,
\end{equation*}
so~$X(\pi)$ and~$X(\rho)$ form a cover relation in the quotient~$S_n/{\equiv}$.
Moreover, when transitioning from the last permutation of~$\lvec{c}(\pi_k)$ to the first permutation of~$\rvec{c}(\pi_{k+1})$, or from the last permutation of~$\rvec{c}(\pi_{k+1})$ to the first permutation of~$\lvec{c}(\pi_{k+2})$, then we move from~$c_1(\pi_k)$ to~$c_1(\pi_{k+1})$, or from~$c_n(\pi_{k+1})$ to~$c_n(\pi_{k+2})$, respectively.
Consequently, as $\pi_k$ and~$\pi_{k+1}$, and also $\pi_{k+1}$ and~$\pi_{k+2}$ form a cover relation in the quotient~$S_{n-1}/{\equiv}^*$ by induction, we obtain with the help of Lemma~\ref{lem:projection} that any two consecutive permutations $\pi,\rho$ in~$J(L_n)$ form a cover relation in the quotient~$S_n/{\equiv}$.
On the other hand, if \eqref{eq:Rn2} holds, then by Lemma~\ref{lem:cong} the set~$R_n$ satisfies condition~(z2), so the permutations of~$L_n=R_n$ are generated in the sequence~$J(L_n)$ defined in~\eqref{eq:JLn2}.
In this case, the claim follows immediately by induction.
\end{proof}

\begin{figure}
\includegraphics[width=\textwidth]{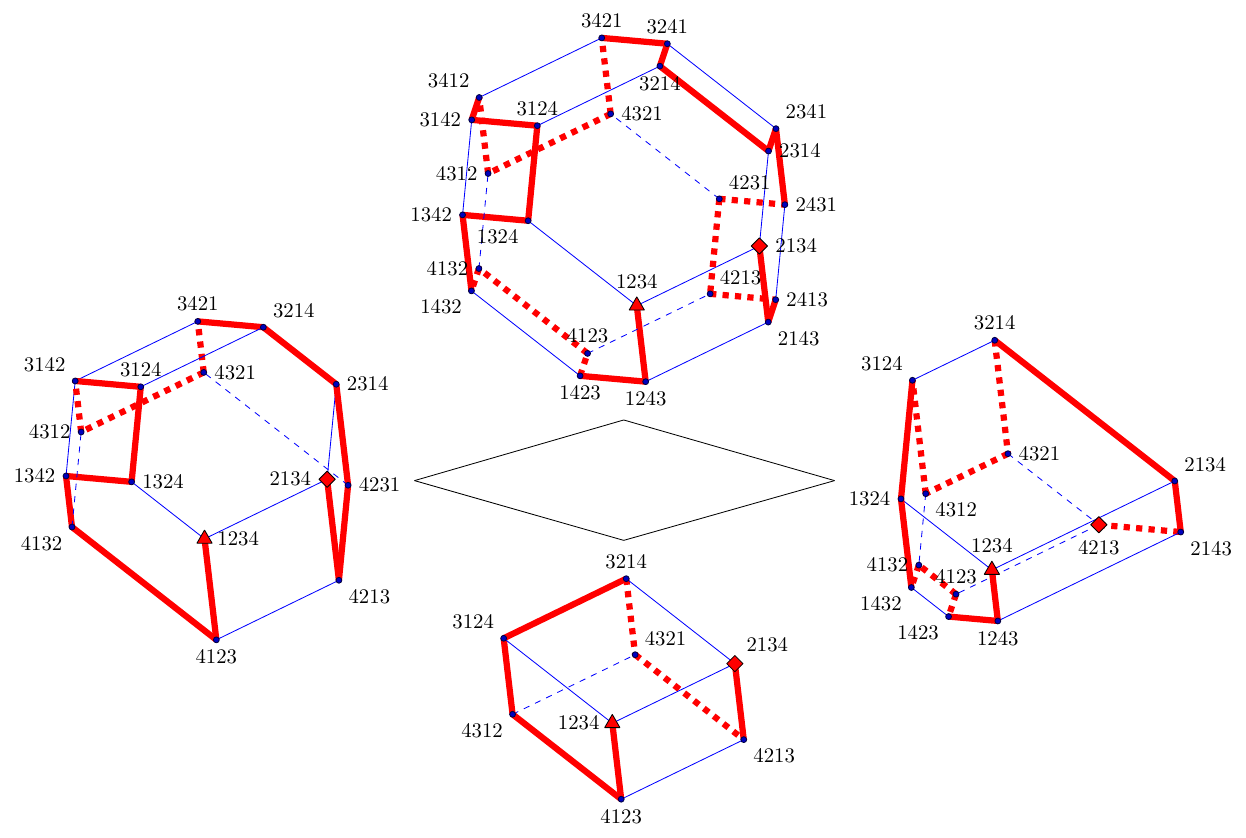}
\caption{Four quotientopes from Figure~\ref{fig:quotient} ordered as a diamond.
These are the permutahedron (top), the associahedron (right), the 3-cube (bottom), and some other polytope (left).
The figure illustrates the consistent choice of representative permutations for the congruence classes, i.e., permutations for lower quotientopes are subsets of permutations for the higher ones.
}
\label{fig:reps}
\end{figure}

Combining Lemmas~\ref{lem:cong} and \ref{lem:jump-order} yields the following theorem.

\begin{theorem}
\label{thm:lattice}
For every lattice congruence~$\equiv$ of the weak order on~$S_n$, let $R_n\seq S_n$ be the set of representatives defined in~\eqref{eq:Rn}.
Then Algorithm~J generates a sequence $J(R_n)=\pi_1,\pi_2,\ldots$ of all permutations from~$R_n$ such that $X(\pi_1),X(\pi_2),\ldots$ is a Hamilton path in the cover graph of the lattice quotient~$S_n/{\equiv}$.
\end{theorem}

For every lattice congruence~$\equiv$, Pilaud and Santos~\cite[Corollary~10]{MR3964495} defined a polytope, called the \emph{quotientope} for~$\equiv$, whose graph is exactly the cover graph of the lattice quotient~$S_n/{\equiv}$.
These polytopes generalize many known polytopes, such as hypercubes, associahedra, permutahedra etc.
The following result is an immediate corollary of Theorem~\ref{thm:lattice}, and it is illustrated in Figure~\ref{fig:quotient}.

\begin{corollary}
\label{cor:quotient}
For every lattice congruence~$\equiv$ of the weak order on~$S_n$, Algorithm~J generates a Hamilton path on the graph of the corresponding quotientope.
\end{corollary}

\begin{remark}
\label{rem:consistent}
Observe that in the definition~\eqref{eq:Rn1}, whenever we encounter an equivalence class~$X\in S_n/{\equiv}$ with $|X\cap r(\pi)|\geq 2$ and $c_1(p(\pi)),c_n(p(\pi))\notin X$, then we have freedom to pick an arbitrary permutation from~$X\cap r(\pi)$ for the set of representatives~$R_\pi$.
By imposing a total order on~$S_n$ (e.g., lexicographic order), we can make these choices unique, and this will make the resulting sets of representatives \emph{consistent} across the entire lattice of congruences ordered by refinement.
Specifically, given two equivalence relations $\equiv$ and~$\equiv'$ where $\equiv$ is a refinement of~$\equiv'$, computing the representatives~$R_n$ and~$R_n'$ according to this rule will result in~$R_n\supseteq R_n'$.
However, the resulting jump ordering~$J(R_n)$ may not be a subsequence of~$J(R_n')$, as argued in~\cite[Remark~3]{MR4391718}.
This consistent choice of representative permutations is illustrated in Figure~\ref{fig:reps}.
\end{remark}

\begin{remark}
\label{rem:cyclic}
In~\cite[Lemma~4]{MR4391718} we showed that if each of the zigzag languages~$R_k$, $2\leq k\leq n-1$, has even cardinality, then the ordering of permutations $J(R_n)$ defined by Algorithm~J is cyclic.
Consequently, if for a given lattice congruence, the number of equivalence classes of each restriction to~$S_k$, $2\leq k\leq n-1$, is even, then Algorithm~J generates a Hamilton \emph{cycle} on the graph of the corresponding quotientope (the converse does not hold in general, but under the additional assumption $|R_2|<|R_3|<\cdots<|R_{n-1}|$).
This happens for instance for the permutahedron and for the hypercube, but not for the associahedron, even though the associahedron is known to admit a Hamilton cycle~\cite{MR920505,MR1723053}.
We are not aware if this condition on the parity of the number of equivalence classes of a lattice congruence can be characterized more easily (e.g., via the arc diagram of the congruence).
We will come back to the question about Hamilton cycles in Section~\ref{sec:open}.
\end{remark}

\section{Regular and vertex-transitive lattice quotients}
\label{sec:reg}

In this section we characterize regular and vertex-transitive quotientopes combinatorially via their arc diagrams, which in particular allows us to count them.
We may either consider these objects in terms of the equivalence classes of the lattice congruence, or in terms of the cover graph of the resulting lattice quotient.
As several congruences may give the same cover graph, the latter distinction is coarser, yielding fewer distinct objects.
Overall, we obtain six different classes of objects, and Table~\ref{tab:trans} summarizes our results for each of them.
The table provides the exact counts for small values of~$n$, various exact and asymptotic counting formulas, as well as references to the theorems where they are established.
In the table, we encounter various familiar counting sequences, namely the squared Catalan numbers, and weighted integer compositions and partitions.
We also establish the precise minimum and maximum degrees for those graph classes, and in the latter result the famous Erd\H{o}s-Szekeres theorem makes its appearance.

\begin{table}
\caption{Number of different classes of quotient graphs that arise from essential lattice congruences and their minimum and maximum degrees.
In this table, $C_n$ denotes the $n$th Catalan number, $c_{n,k}$ denotes the number of integer compositions of~$n$ with exactly $k$ many~2s, and $t_n$ denotes the number of~2s in all integer partitions of~$n$.
The last column contains references to the corresponding sequence numbers in the OEIS~\cite{oeis}.
}
\label{tab:trans}
\setlength{\tabcolsep}{5pt}
\makebox[0cm]{ 
\begin{tabular}{l|l|r@{\hskip 2pt}r@{\hskip 2pt}r@{\hskip 2pt}r@{\hskip 2pt}r@{\hskip 2pt}r|c|c|l}
                     & \multicolumn{2}{r}{$n=2\!$} & 3 & 4 & 5 & 6 & 7 & General formulas/bounds & Ref. & OEIS \\ \hline
quotient graphs      & $|\cQ_n|$  &     1 & 4 & 47 & {\footnotesize 3.322} & {\footnotesize 11.396.000} & ? & $[2^{2^{n-2}}, 2^{2^n-2n}]$ & Thm.~\ref{thm:count-Qn} & A330039 \\
regular              & $|\cR_n|$  & 1 & 4 & 25 & 196 & 1.764 & {\footnotesize 17.424} & $=C^2_{n-1}=16^{n(1+o(1))}$ & Cor.~\ref{cor:count-Rn} & A001246 \\
vertex-transitive    & $|\cV_n|$  & 1 & 4 &  8 & 22 & 52 & 132 & $=\sum\limits_{k\geq 0} 3^k c_{n-1,k}=2.48...^{n(1+o(1))}$ & Cor.~\ref{cor:count-Vn} & A052528 \\
non-iso.\            & $|\cQ_n'|$ & 1 & 3 & 19 & 748 & {\footnotesize 2.027.309} & ? & $\geq 2^n - 2n +1$ & Thm.~\ref{thm:count-Qn'} & A330040 \\
non-iso.\ regular    & $|\cR_n'|$ & 1 & 3 & 10 & 51 & 335 & {\footnotesize 2.909} & ? & & A330042 \\
non-iso.\ vertex-tr. & $|\cV_n'|$ & 1 & 3 &  4 & 8 & 11 & 19 & $=t_{n+1}=e^{\pi\sqrt{2n/3}(1+o(1))}$ & Cor.~\ref{cor:count-Vn'} & A024786 \\
\hline
minimum degree       &            & 1 & 2 & 3 & 4 & 5 & 6 & $=n-1$ & Thm.~\ref{thm:min-deg} & \\
maximum degree       &            & 1 & 2 & 4 & 5 & 7 & 8 & $=2n-\lceil 2\sqrt{n}\rceil$ & Thm.~\ref{thm:max-deg} & A123663 \\
\end{tabular}
}
\end{table}

\subsection{Preliminaries}

We let $\cC_n$ denote the set of all lattice congruences of the weak order on~$S_n$.
Throughout this section, we will denote lattice congruences by capital Latin letters such as~$R\in \cS_n$, and whenever we consider two permutations~$\pi,\rho$ in the same equivalence class of $S_n/R$, we write $\pi\equiv_R \rho$ or simply $\pi\equiv \rho$, if $R$ is clear from the context.
Recall from Theorem~\ref{thm:reading} that every lattice congruence~$R\in \cC_n$ corresponds to a downset $F_R\seq F_n$ of fences in the forcing order, and that such a downset can be represented by its arc diagram, which contains exactly one arc for each fence from~$F_R$.
The \emph{reduced arc diagram} contains only the arcs that correspond to maximal elements in the downset~$F_R$, i.e., to fences that are pairwise incomparable in the forcing order.
Every fence not of the form~$f(a,a+1,\emptyset)$, $a\in[n-1]$, is referred to as \emph{essential}, and we let $F_n^*\seq F_n$ denote the set of all essential fences.
We refer to any lattice congruence~$R$ with $F_R\seq F_n^*$ as \emph{essential}, and we let $\cC_n^*\seq \cC_n$ denote the set of all essential lattice congruences.
Note that by this definition, the arc diagrams of essential lattice congruences do not contain any arcs that connect consecutive points~$a$ and~$a+1$, $a\in[n-1]$.

We refer to the underlying undirected graph of the cover graph of any lattice quotient $S_n/R$, $R\in \cC_n$, as a \emph{quotient graph}~$Q_R$, and we define $\cQ_n:=\{Q_R\mid R\in\cC_n^*\}$.

All 47 essential lattice congruences~$\cC_n^*$ for $n=4$ are shown in Figure~\ref{fig:arcs}, ordered by refinement of the congruences and represented by their arc diagrams, where the arcs of the reduced diagrams are highlighted.
Recall from the previous section that for every essential lattice congruence~$R\in \cC_n^*$, Pilaud and Santos~\cite[Corollary~10]{MR3964495} defined an $(n-1)$-dimensional polytope, called the \emph{quotientope} of~$R$, whose graph is exactly the quotient graph~$Q_R$.
These polytopes are shown in Figure~\ref{fig:quotient}, where the regular and vertex-transitive graphs are marked with~* and~**, respectively.

The following lemma justifies that in our definition of~$\cC_n^*$, we exclude fences that are not essential.
The reason is that including them results in a dimension collapse, i.e., the resulting lattice quotient is isomorphic to some quotient of smaller dimension; see Figure~\ref{fig:dim}.

Given two posets~$(P,<_P)$ and $(Q,<_Q)$, the \emph{Cartesian product} is the poset~$(P\times Q,<)$ with $(p,q)<(p',q')$ if and only if $p<_P p'$ and $q<_Q q'$.
For any set of fences~$F\seq F_n$ and any interval~$[s,t]$, $1\leq s\leq t\leq n$, we define $F|_{[s,t]}:=\{f(a,b,L)\in F\mid s\leq a<b\leq t\}$, i.e., we select all fences from $F$ that lie entirely in this interval.
Moreover, for any integer~$s$ we define $F+s:=\{f(a+s,b+s,L+s)\mid f(a,b,L)\in F\}$ with $L+s:=\{x+s\mid x\in L\}$, i.e., we shift all fences by~$s$.

\begin{lemma}
\label{lem:dim}
Let $R\in \cC_{n+1}$ be a lattice congruence with a non-essential fence~$f(s,s+1,\emptyset)\in F_R$, and define lattice congruences $A\in\cC_s$ and $B\in\cC_{n+1-s}$ by $F_A=F_R|_{[1,s]}$ and $F_B=F_R|_{[s+1,n+1]}-s$.
Moreover, let $R'\in \cC_n$ be the lattice congruence given by
\begin{equation}
\label{eq:FR'}
F_{R'}=F_A\cup \big(F_B+(s-1)\big)\cup D,
\end{equation}
where $D$ is the downset of the fences~$f(s-1,s+1,\emptyset)$ and~$f(s-1,s+1,\{s\})$ in the forcing order for~$S_n$.
Then $S_{n+1}/R$ and~$S_n/R'$ are both isomorphic to the Cartesian product of~$S_s/A$ and~$S_{n+1-s}/B$.
In particular, the lattice quotients~$S_{n+1}/R$ and~$S_n/R'$ are isomorphic.
\end{lemma}

\begin{figure}
\includegraphics{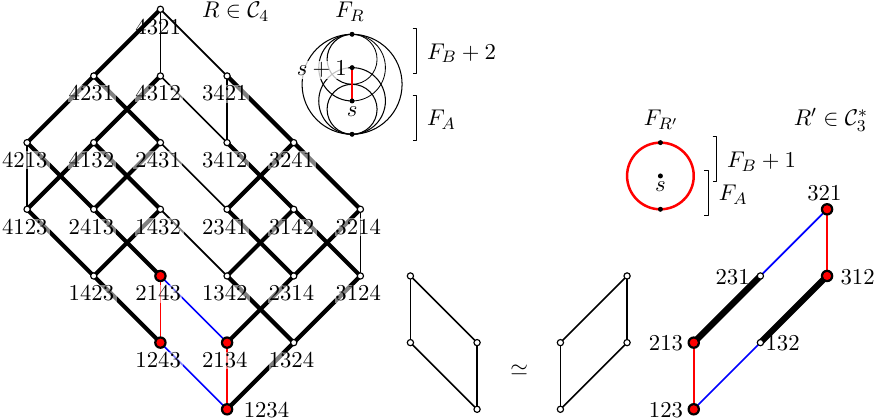}
\caption{Illustration of Lemma~\ref{lem:dim}.
The left hand side shows the lattice congruence from~$\cC_4$ given by the downset of the non-essential fence~$f(2,3,\emptyset)$.
The right hand side shows the lattice congruence from~$\cC_3^*$ given by the downset of the essential fences~$\{f(1,3,\emptyset),f(1,3,\{2\})\}$.
Both lattice quotients are isomorphic to the Cartesian product of~$S_2$ and~$S_2$, whose cover graph is a 4-cycle.
}
\label{fig:dim}
\end{figure}

\begin{proof}
Consider two equivalence classes~$X$ and~$Y$ of~$R$, and two permutations~$\pi\in X$ and~$\rho\in Y$ that differ in an adjacent transposition of two entries~$a$ and~$b$.
As $F_R$ contains the fence~$f(s,s+1,\emptyset)$, the definition of forcing order implies that $F_R$ also contains all fences~$f(c,d,L)$ for all $c\in [1,s]$, $d\in [s+1,n+1]$ and~$L\seq \left]c,d\right[$.
This means there are permutations~$\pi_0\in X$ and~$\rho_0\in Y$, such that in~$\pi_0$ and~$\rho_0$ all entries from~$[1,s]$ appear before all entries from~$[s+1,n+1]$, and $\pi_0$ and~$\rho_0$ differ in an adjacent transposition of~$a$ and~$b$, and either $a,b\in [1,s]$ or $a,b\in [s+1,n+1]$.
We can reach $\pi_0$ and~$\rho_0$ from~$\pi$ and~$\rho$, respectively, by moving down within the equivalence classes~$X$ or~$Y$ towards permutations with fewer inversions, repeatedly swapping any entry from~$[1,s]$ that is to the right of any entry from~$[s+1,n+1]$.
It follows that every cover relation of~$S_{n+1}/R$ has a corresponding cover relation in the Cartesian product of~$S_s/A$ and~$S_{n+1-s}/B$.

Consider two equivalence classes~$X$ and~$Y$ of~$R'$, and two permutations~$\pi\in X$ and~$\rho\in Y$ that differ in an adjacent transposition of two entries~$a$ and~$b$.
By the definition~\eqref{eq:FR'}, the set $F_{R'}$ contain the fences $f(s-1,s+1,\emptyset)$, $f(s-1,s+1,\{s\})$, and all fences in their downset of the forcing order for~$S_n$, so the definition of forcing order yields that $F_{R'}$ also contains all fences~$f(c,d,L)$ for all $c\in [1,s-1]$, $d\in[s+1,n]$ and~$L\seq \left]c,d\right[$.
It follows that either $a,b\in [1,s]$ or $a,b\in [s,n]$.
In the first case, there are permutations~$\pi_0\in X$ and~$\rho_0\in Y$, such that in~$\pi_0$ and~$\rho_0$ all entries from~$[1,s]$ appear at consecutive positions, surrounded by all entries from~$[s+1,n]$, and~$\pi_0$ and~$\rho_0$ differ in an adjacent transposition of~$a$ and~$b$.
In the second case, there are permutations~$\pi^0\in X$ and~$\rho^0\in Y$, such that in~$\pi^0$ and~$\rho^0$ all entries from~$[s,n]$ appear at consecutive positions, surrounded by all entries from~$[1,s-1]$, and $\pi^0$ and~$\rho^0$ differ in an adjacent transposition of~$a$ and~$b$.
Morever, as $\pi,\pi_0,\pi^0\in X$ and $\rho,\rho_0,\rho^0\in Y$, we obtain that every cover relation of~$S_n/R'$ has a corresponding cover relation in the Cartesian product of~$S_s/A$ and~$S_{n+1-s}/B$.
\end{proof}

Given any lattice congruence $R\in\cC_n$ for which $F_R$ contains non-essential fences, we may repeatedly apply Lemma~\ref{lem:dim} to eliminate them, until we arrive at a lattice congruence~$R'\in\cC_m^*$, $m<n$, with an isomorphic quotient graph $Q_{R'}\simeq Q_R$.

\subsection{Exact counts for small dimensions}

With computer help, we determined the number of essential lattice congruences, or equivalently, the number of quotient graphs, for $2\leq n\leq 6$.
The results are shown in Table~\ref{tab:trans}.
We also computed the sets $\cR_n\seq \cQ_n$ and $\cV_n\seq \cQ_n$ of all regular and vertex-transitive quotient graphs, respectively, for $2\leq n\leq 7$, with the help of Theorem~\ref{thm:regular}.

Many of the quotient graphs from~$\cQ_n$ are isomorphic; cf.~\cite[Figure~8]{MR3964495}.
This happens for instance if the corresponding arc diagrams differ only by rotation of reflection, but not only in this case; see Figure~\ref{fig:iso}.
To this end, we let $\cQ_n'$ denote all non-isomorphic quotient graphs from~$\cQ_n$, and we let $\cR_n'$ and $\cV_n'$ be the non-isomorphic regular and vertex-transitive ones.
The corresponding counts for small~$n$ are also shown in Table~\ref{tab:trans}.
We clearly have $\cV_n\seq \cR_n\seq \cQ_n$ and $\cV_n'\seq \cR_n'\seq \cQ_n'$.

\subsection{Counting quotient graphs}
\label{sec:count-Qn}

The following theorem shows that there are double-exponentially many quotient graphs.

\begin{theorem}
\label{thm:count-Qn}
For all $n\geq 3$, we have $2^{2^{n-2}}\leq|\cQ_n|\leq 2^{2^n-2n}$.
\end{theorem}

\begin{proof}
The number of fences~$f(a,b,L)\in F_n$ with $b-a=k\in\{1,\ldots,n-1\}$ is exactly~$f_k:=(n-k)2^{k-1}$, as for fixed~$k$, there are $(n-k)$ different choices for~$a$ and~$b$, and for fixed~$a$ and~$b$, there are $2^{k-1}$ many choices for~$L\seq \left]a,b\right[$.
As all fences with $k=n-1$ are essential for $n\geq 3$ and also incomparable in the forcing order, we obtain at least $2^{f_{n-1}}$ distinct downsets.
The total number of essential fences is~$\sum_{k=2}^{n-1}f_k=2^n-2n=:s$, so there are at most $2^s$ distinct downsets.
\end{proof}

To estimate the cardinality of~$\cQ_n'$, we have to factor out symmetries of the arc diagrams, i.e., horizontal and vertical reflections, which account for a factor of at most 4.
However, isomorphic graphs also arise from arc diagrams that do not only differ by those symmetries; see Figure~\ref{fig:iso}.
In particular, we have $|\cQ_n|/|\cQ_n'|>4$ for $n=5$ and $n=6$; see Table~\ref{tab:trans}.

\begin{figure}
\begin{tikzpicture}[scale=.3]
\draw[maxarc] (0,6) arc(90:-90:2.0);
\draw (0,3) arc(90:270:1.5);
\draw (0,6) arc(90:-90:1.5);
\draw (0,6) arc(90:-90:3.0);
\draw[maxarc] (0,4) arc(90:270:2.0);
\draw (0,6) arc(90:270:3.0);
\node[vertex](T0) at (0,0) {};
\node[vertex](T2) at (0,2) {};
\node[vertex](T4) at (0,4) {};
\node[vertex](T6) at (0,6) {};
\end{tikzpicture}
\hspace{2mm}
\begin{tikzpicture}[scale=.3]
\draw[maxarc] (0,4) arc(90:-90:2.0);
\draw (0,3) arc(90:270:1.5);
\draw (0,6) arc(90:-90:1.5);
\draw (0,6) arc(90:-90:3.0);
\draw[maxarc] (0,6) arc(90:-90:2.0);
\draw (0,3) arc(90:-90:1.5);
\draw (0,6) arc(90:270:1.5);
\node[vertex](T0) at (0,0) {};
\node[vertex](T2) at (0,2) {};
\node[vertex](T4) at (0,4) {};
\node[vertex](T6) at (0,6) {};
\draw (6,0) -- (6,8);
\end{tikzpicture}
\hspace{5mm}
\begin{tikzpicture}[scale=.3]
\draw[maxarc] (0,3) arc(90:270:1.5);
\draw[maxarc] (0,3) arc(-90:90:2.5);
\node[vertex](T0) at (0,0) {};
\node[vertex](T2) at (0,2) {};
\node[vertex](T4) at (0,4) {};
\node[vertex](T6) at (0,6) {};
\node[vertex](T8) at (0,8) {};
\end{tikzpicture}
\hspace{2mm}
\begin{tikzpicture}[scale=.3]
\draw[maxarc] (0,3) arc(90:270:1.5);
\draw[maxarc] (0,3) arc(-90:90:1);
\draw[maxarc] (0,5) arc(270:90:1.5);
\node[vertex](T0) at (0,0) {};
\node[vertex](T2) at (0,2) {};
\node[vertex](T4) at (0,4) {};
\node[vertex](T6) at (0,6) {};
\node[vertex](T8) at (0,8) {};
\draw (5,0) -- (5,8);
\end{tikzpicture}
\hspace{5mm}
\begin{tikzpicture}[scale=.3]
\draw[maxarc] (0,2.5) arc(90:270:0.75);
\draw[maxarc] (0,4.5) arc(90:-90:1);
\draw[maxarc] (0,4.5) arc(270:90:1);
\draw[maxarc] (0,8) arc(90:-90:0.75);
\node[vertex](T1) at (0,1) {};
\node[vertex](T2) at (0,2) {};
\node[vertex](T3) at (0,3) {};
\node[vertex](T4) at (0,4) {};
\node[vertex](T5) at (0,5) {};
\node[vertex](T6) at (0,6) {};
\node[vertex](T7) at (0,7) {};
\node[vertex](T8) at (0,8) {};
\end{tikzpicture}
\hspace{2mm}
\begin{tikzpicture}[scale=.3]
\draw[maxarc] (0,2.5) arc(90:270:0.75);
\draw[maxarc] (0,3.5) arc(90:-90:0.5);
\draw[maxarc] (0,3.5) arc(270:90:0.5);
\draw[maxarc] (0,5.5) arc(90:-90:0.5);
\draw[maxarc] (0,5.5) arc(270:90:0.5);
\draw[maxarc] (0,8) arc(90:-90:0.75);
\node[vertex](T1) at (0,1) {};
\node[vertex](T2) at (0,2) {};
\node[vertex](T3) at (0,3) {};
\node[vertex](T4) at (0,4) {};
\node[vertex](T5) at (0,5) {};
\node[vertex](T6) at (0,6) {};
\node[vertex](T7) at (0,7) {};
\node[vertex](T8) at (0,8) {};
\end{tikzpicture}
\caption{Three pairs of lattice congruences from~$\cC_4^*$ (left), $\cC_5^*$ (middle), and~$\cC_8^*$ (right), with distinct arc diagrams but isomorphic quotient graphs.}
\label{fig:iso}
\end{figure}
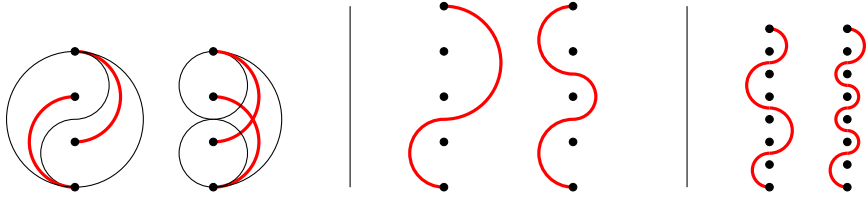

This difference in the growth rates can partially be explained by arc diagrams that induce a graph product structure.
I.e., if we have an arc diagram with two arcs corresponding to the fences~$f(s-1,s+1,\emptyset)$ and~$f(s-1,s+1,\{s\})$, then by Lemma~\ref{lem:dim} the two parts of the arc diagram separated by these two fences can be mirrored independently, or modified as described by Figure~\ref{fig:iso}, yielding the same resulting quotient graph.
Such operations clearly yield many more than 4 symmetries.
We cannot fully explain this, but we provide the following lower bound.

\begin{theorem}
\label{thm:count-Qn'}
For all $n\geq 3$, we have $|\cQ_n'|\geq 2^n-2n+1$.
\end{theorem}

\begin{proof}
We argued before that the total number of essential fences in the forcing order is~$2^n-2n=:s$.
This implies that the lattice of congruences ordered by refinement (see Figure~\ref{fig:arcs}) contains a chain $R_0,\ldots,R_s\in\cC_n^*$ of size~$s+1$, where $R_0$ is the maximal element and $R_s$ is the minimal element, and along this chain we have $|F_{R_i}|=i$ for $i=0,\ldots,s$.
Consequently, the number of vertices of the quotient graphs $Q_{R_i}$, $i=0,\ldots,s$, forms a strictly decreasing sequence, starting with $n!$ and ending with $2^{n-1}$.
This is because whenever an additional fence is added, the equivalence classes grow, and so the quotient graph shrinks.
In particular, all those quotient graphs are non-isomorphic, proving that $|\cQ_n'|\geq s+1$.
\end{proof}

\subsection{Regular quotient graphs}
\label{sec:count-Rn}

It turns out that the regular quotient graphs~$\cR_n$ can be characterized and counted precisely via their arc diagrams.
Specifically, we say that an arc is \emph{simple} if it does not connect two consecutive points and if it does not cross the vertical line.
Also, we say that a reduced arc diagram is \emph{simple} if it contains only simple arcs.
Note that the fence $f(a,b,L)$ corresponding to a simple arc either satisfies $L=\emptyset$ or $L=\left]a,b\right[$.
For example, in Figure~\ref{fig:iso}, the leftmost two reduced arc diagrams are simple, whereas the others are not.
In Theorem~\ref{thm:regular} below, we establish that a quotient graph is regular if and only if the corresponding reduced arc diagram is simple.
This yields a closed counting formula involving the squared Catalan numbers; see Corollary~\ref{cor:count-Rn}.

The first lemma allows us to compute degrees of the quotient graph by considering only the minima and maxima of equivalence classes.

\begin{lemma}
\label{lem:neighbor}
Let~$X$ be an equivalence class of a lattice congruence~$R\in\cC_n$.
Consider all descents in $\pi:=\min(X)$ and all permutations $\pi_1',\ldots,\pi_d'$ obtained from~$\pi$ by transposing one of them.
Also, consider all ascents in $\rho:=\max(X)$ and all permutations $\rho_1',\ldots,\rho_a'$ obtained from~$\rho$ by transposing one of them.
Then the down-neighbors of~$X$ in the quotient graph~$Q_R$ are $X(\pi_1'),\ldots,X(\pi_d')$, and they are all distinct, and the up-neighbors of~$X$ in the quotient graph are $X(\rho_1'),\ldots,X(\rho_a')$, and they are all distinct.
In particular, the degree of~$X$ in the quotient graph is the number of descents of~$\min(X)$ plus the number of ascents of~$\max(X)$.
\end{lemma}

\begin{proof}
From Lemma~\ref{lem:interval} it follows that for any lattice congruence, the down-neighbors of the minimum of an equivalence class~$X$ all belong to distinct equivalence classes, and the up-neighbors of the maximum of~$X$ all belong to distinct equivalence classes.
Recall that in the weak order on~$S_n$, the down-neighbors of a vertex are reached by adjacent transpositions of descents, and the up-neighbors are reached by adjacent transpositions of ascents.
From this the statement follows with the help of Lemma~\ref{lem:subposet}.
\end{proof}

The next lemma helps us to compute the maximum of an equivalence class quickly.
It is an immediate consequence of the definition of forcing order.

\begin{lemma}
\label{lem:swap-blocks}
Consider a lattice congruence $R\in\cC_n$ and a permutation $\pi$ with an ascent~$(a,b)$.
Let $A$ be a substring ending with~$a$ of entries of~$\pi$ of size at most~$a$, and $B$ a substring starting with~$b$ of entries that are of size at least~$b$, i.e., we have $\pi=L\,A\,B\,R$ for some substrings $L,R$.
If the permutation~$\rho$ obtained by transposing the pair~$(a,b)$ is in the same equivalence class as~$\pi$, i.e., $\pi\equiv \rho$, then they are also in the same equivalence class as the permutation obtained by swapping the entire substrings~$A$ and~$B$, i.e., $\pi\equiv\rho\equiv L\,B\,A\,R$.
\end{lemma}

\begin{wrapfigure}{r}{0.45\textwidth}
\flushright
\includegraphics{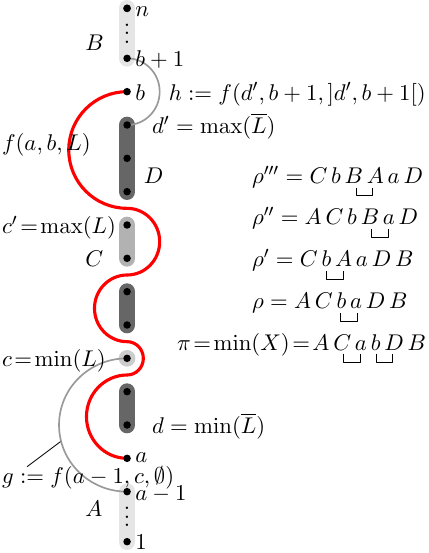}
\captionsetup{width=.85\linewidth}
\caption{Illustration of the proof of Lemma~\ref{lem:non-simple}.
Descents in the permutations are marked by square brackets.
}
\vspace{-8mm}
\label{fig:non-simple}
\end{wrapfigure}

There is a corresponding version of Lemma~\ref{lem:swap-blocks} for swapping substrings around a descent~$(b,a)$, to quickly compute the minimum of an equivalence class, but we omit stating this symmetric variant explicitly here.

We first rule out non-simple arc diagrams as candidates for giving a regular quotient graph.

\begin{lemma}
\label{lem:non-simple}
If an essential lattice congruence~$R\in\cC_n^*$ has a non-simple arc in its reduced arc diagram, then the quotient graph~$Q_R$ is not regular.
\end{lemma}

\begin{proof}
As $R$ is essential, $F_R$ does not contain any fence of the form~$f(s,s+1,\emptyset)$, $s\in[n-1]$.
Consequently, the equivalence class containing the identity permutation~$\ide_n$ does not contain any other permutations and so has degree~$n-1$ in~$Q_R$ by Lemma~\ref{lem:neighbor}.
In the following we identify an equivalence class~$X$ whose degree in~$Q_R$ is~$n$, which proves that $Q_R$ is not a regular graph.
This part of the proof is illustrated in Figure~\ref{fig:non-simple}.

Consider a non-simple arc in the reduced arc diagram of~$R$, and consider the corresponding fence~$f(a,b,L)\in F_R$, $a<b$.
We define $\ol{L}:=\left]a,b\right[\setminus L$.
The assumption that the arc is not simple means that $L$ and $\ol{L}$ are both non-empty.
We define $c:=\min(L)$, $c':=\max(L)$, $d:=\min(\ol{L})$, and $d':=\max(\ol{L})$, and we write~$C$ and~$D$ for the increasing sequences of numbers in the sets~$L$ and~$\ol{L}$, respectively.
We also define the sequences $A:=(1,\ldots,a-1)$ and $B:=(b+1,\ldots,n)$.
Now consider the equivalence class~$X\in S_n/R$ which contains the permutations $\pi:=A\,C\,a\,b\,D\,B$ and $\rho:=A\,C\,b\,a\,D\,B$.

Clearly, $\pi$ and~$\rho$ differ in an adjacent transposition of the entries~$a$ and~$b$, and we have $\pi\equiv \rho$ due to the fence~$f(a,b,L)\in F_R$.
We first prove that $\pi=\min(X)$.
By Lemma~\ref{lem:neighbor}, we need to check the descents of~$\pi$, and there are exactly two of them, namely~$(c',a)$ and~$(b,d)$.
None of them can be transposed to reach a permutation in~$X$, as neither the fence $f(a,c',\left]a,c'\right[\cap L)$, nor the fence $f(d,b,\left]d,b\right[\cap L)$ is in~$F_R$, as both are above $f(a,b,L)$ in the forcing order, and if one of them was in~$F_R$, then the arc corresponding to $f(a,b,L)$ would not be in the reduced arc diagram.
This proves that $\pi$ is the minimum of~$X$.

Now consider the permutation~$\rho$.
It has only one descent~$(b,a)$, and so $n-1$ ascents.
The ascents~$(c',b)$ and~$(a,d)$ in~$\rho$ cannot be transposed to reach a permutation in~$X$, as neither the fence $f(c',b,\left]c',b\right[\cap L)$ nor the fence $f(a,d,\left]a,d\right[\cap L)$ is in~$F_R$, as both are above $f(a,b,L)$ in the forcing order.
Similarly, the ascents that lie entirely within~$C$ or~$D$ cannot be transposed, as for any such ascent~$(r,s)$, the corresponding fence $f(r,s,\left]r,s\right[\cap L)$ is above $f(a,b,L)$ in the forcing order.
Moreover, none of the ascents~$(s,s+1)$ that lie entirely within~$A$ or~$B$ can be transposed, as $F_R$ is essential by assumption and so contains none of the fences $f(s,s+1,\emptyset)$, $s\in[n-1]$.
It remains to consider the ascents~$(a-1,c)$ and~$(d',b+1)$.
They can possibly be transposed to reach a permutation in~$X$, but only if the fence~$g:=f(a-1,c,\emptyset)$ or the fence $h:=f(d',b+1,\left]d',b+1\right[)$ is in~$F_R$, which may or may not be the case.
If $g\notin F_R$ and $h\notin F_R$, then we have $\rho=\max(X)$, and so $\desc(\max(X))=1$.
If $g\in F_R$ and $h\notin F_R$, then Lemma~\ref{lem:swap-blocks} shows that $\rho':=C\,b\,A\,a\,D\,B=\max(X)$, and again we get $\desc(\max(X))=1$, as the only descent in $\rho'$ is~$(b,1)$. 
If $g\notin F_R$ and $h\in F_R$, then Lemma~\ref{lem:swap-blocks} shows that $\rho'':=A\,C\,b\,B\,a\,D=\max(X)$, and again we get $\desc(\max(X))=1$, as the only descent in $\rho''$ is~$(n,a)$.
If $g\in F_R$ and $h\in F_R$, then Lemma~\ref{lem:swap-blocks} shows that $\rho''':=C\,b\,B\,A\,a\,D=\max(X)$, and again we get $\desc(\max(X))=1$, as the only descent in $\rho'''$ is~$(n,1)$.

We have shown that $\desc(\min(X))=2$ and $\desc(\max(X))=1$, and so $\asc(\max(X))=n-2$.
Therefore, by Lemma~\ref{lem:neighbor}, the degree of~$X$ in~$Q_R$ is $\desc(\min(X))+\asc(\max(X))=2+(n-2)=n$, which shows that $Q_R$ is not regular.
This completes the proof.
\end{proof}

We now aim to prove that a simple reduced arc diagram implies an $(n-1)$-regular quotient graph.
For this we need two auxiliary lemmas.

\begin{lemma}
\label{lem:minmax-boundary}
Consider a lattice congruence $R\in\cC_n$ and an equivalence class~$X$ such that $\min(X)$ has $n$ at the rightmost position, and $\max(X)=\cdots c\,n\,d\cdots$, where $c,d\in[n-1]$.
Then $(c,d)$ is a descent in $p(\max(X))\in S_{n-1}$.
Similarly, suppose that $\max(X)$ has $n$ at the leftmost position, and $\min(X)=\cdots a\,n\,b\cdots$, where $a,b\in[n-1]$.
Then $(a,b)$ is an ascent in $p(\min(X))\in S_{n-1}$.
\end{lemma}

\begin{proof}
Consider an equivalence class~$X$ such that $\pi:=\min(X)$ has $n$ at the rightmost position, and $\rho:=\max(X)=\cdots c\,n\,d\cdots$, where $c,d\in[n-1]$.
Let $\pi':=p(\pi)\in S_{n-1}$ and $\rho':=p(\rho)\in S_{n-1}$.
By Lemma~\ref{lem:restrict-minmax}, $\pi'$ and~$\rho'$ are the minimum and maximum of the equivalence class~$p(X)$ of the restriction~$R^*$, and as $\pi=c_n(\pi')=\pi'\,n$, we also have that $\sigma:=c_n(\rho')=\rho'\,n\in X$.
Note that $\rho$ is obtained from~$\sigma$ by moving $n$ to the left until the entry~$c$ is directly left of it.
In particular, the down-neighbor~$\tau$ of~$\rho$ obtained by transposing~$n$ and~$d$ satisfies $\tau\equiv \rho$ (recall Lemma~\ref{lem:interval}).
By Lemma~\ref{lem:swap-blocks}, it follows that $(c,d)$ must be a descent in~$\tau$, as otherwise, the up-neighbor~$\tau'$ of~$\rho$ obtained by transposing~$c$ and~$n$ would also satisfy $\tau'\equiv \rho$, contradicting the fact that $\rho$ is the maximum of~$X$.
Consequently $(c,d)$ is also a descent in~$\sigma=\rho'\,n$ and in $\rho'=p(\max(X))$.
The proof of the second part of the lemma is analogous.
\end{proof}

For any permutation $\pi\in S_n$ and any integers $a,b\in [n]$ with $a<b$, we let $L(\pi,a,b)$ denote the set of all entries of~$\pi$ that are to the left of~$a$, and whose values are in the interval~$]a,b[$.
For example, for $\pi=817632459$ we have $L(\pi,1,3)=\emptyset$, $L(\pi,2,5)=\{3\}$, and $L(\pi,3,7)=\{6\}$.

\begin{lemma}
\label{lem:minmax-middle}
Consider a lattice congruence $R\in\cC_n$ and an equivalence class~$X$ such that $\pi:=\min(X)=\cdots a\,n\,b\cdots$ and $\rho:=\max(X)=\cdots c\,n\,d\cdots$, where $a,b,c,d\in[n-1]$.
Let $\pi'$ be the last permutation in~$X$ obtained from~$\pi$ by moving~$n$ to the left, and let $\rho'$ be the last permutation in~$X$ obtained from $\rho$ by moving~$n$ to the right.
Then we have the following:
\begin{enumerate}[label=(\roman*), leftmargin=8mm, noitemsep, topsep=3pt plus 3pt]
\item the entry left of~$n$ in $\pi'$ is~$c$;
\item the entry right of~$n$ in~$\rho'$ is~$b$;
\item if $c\neq a$, then $c$ is to the left of~$a$ in~$\pi$, and we have $c,b>x$ for all $x$ between~$c$ and~$n$ in~$\pi$, in particular $a<b$;
\item if $b\neq d$, then $b$ is to the right of~$d$ in~$\rho$, and we have $c,b>x$ for all $x$ between~$n$ and~$b$ in~$\rho$, in particular $c>d$;
\item for any entry~$x$ between~$c$ and~$n$ in~$\pi$, the fence $f(x,n,L(\pi,x,n))$ is in~$F_R$;
\item for any entry~$x$ between~$n$ and~$b$ in~$\rho$, the fence $f(x,n,L(\rho,x,n))$ is in~$F_R$;
\item the fences $f(b,n,L(\pi,b,n))$ and $f(x,c,L(\pi,x,c))$, where $x$ is the entry right of~$c$ in~$\pi$, are not in~$F_R$;
\item the fences $f(c,n,L(\rho,c,n))$ and $f(x,b,L(\rho,x,b))$, where $x$ is the entry left of~$b$ in~$\rho$, are not in~$F_R$.
\end{enumerate}
\end{lemma}

\begin{figure}
\includegraphics{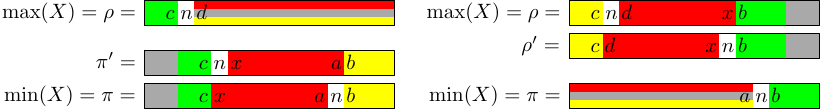}
\caption{Illustration of Lemma~\ref{lem:minmax-middle}.}
\label{fig:minmax}
\end{figure}

\begin{proof}
We only need to prove~(i), (iii), (v) and~(vii), as the other four statements are symmetric.

We first prove~(i).
Suppose for the sake of contradiction that this is not the case and that there is another entry~$e\neq c$ left of~$n$ in~$\pi'$.
Note that $e$ must be to the right of~$c$ in~$\pi'$ and $\pi$, as otherwise $(n,c)$ would be an inversion in $\pi'$, but not in $\rho$, contradicting the fact that $\rho$ is the maximum of~$X$.
We let $\pi''\notin X$ be the permutation obtained from $\pi'$ by transposing~$e$ and $n$, i.e., $\inv(\pi'')=\inv(\pi')\cup\{(n,e)\}$.
It follows that $(n,e)\notin\inv(\rho)$, as otherwise $\pi''$ would be contained in the interval~$[\pi,\rho]=X$.
This means that $e$ is left of~$n$ and left of~$c$ in~$\rho$.
This implies that $c<e$, as otherwise $(c,e)$ would be an inversion in~$\pi'$, but not in~$\rho$.

We now move up in the weak order from~$\pi'$ to~$\rho$, creating a sequence of permutations that all contain the ascent~$(e,n)$, as follows:
Starting from~$\pi'$, we repeatedly choose an arbitrary ascent that we can transpose to stay inside~$X$.
First note that $(n,x)$, $x\in[n-1]$, is never an ascent, so $n$ never moves to the right.
Also $(e,n)$ can never be transposed, as $(n,e)\notin\inv(\rho)$.
Whenever we encounter a transposition that involves the entry~$e$, then it must be a transposition of the form $(e',e)\rightarrow (e,e')$ with $e'<e$ (otherwise the inversion set would shrink), and we then immediately also perform the transposition $(e',n)\rightarrow (n,e')$, keeping~$e$ and~$n$ next to each other, and reaching a permutation in~$X$ by Lemma~\ref{lem:swap-blocks}.
As $(c,e)\notin\inv(\pi')$ and $(c,e)\in\inv(\rho)$, we must encounter a step in which~$c$ and~$e$ are transposed.
However, by the same reasoning, we can then also transpose~$c$ and~$n$ to reach another permutation in~$X$, a contradiction to $(n,c)\notin\inv(\rho)$.
This completes the proof of~(i).

We now prove~(iii).
The fact that $c$ is left of~$a$ in~$\pi$ is an immediate consequence of~(i).
Note that $\pi'$ is obtained from~$\pi$ by moving~$n$ to the left until $n$ is directly right of~$c$, and $n$ cannot move further as $(n,c)\notin\inv(\rho)$.
Conversely, $\pi$ is obtained from~$\pi'$ by moving~$n$ to the right until $n$ is directly left of~$b$, and $n$ cannot move further as $\pi$ is the minimum of~$X$.
In particular, the entry~$n$ can be moved across the largest entry~$e$ between~$c$ and~$n$ in~$\pi$.
Lemma~\ref{lem:swap-blocks} therefore shows that $b,c>e$, as otherwise $n$ could move to the left of~$c$ in~$\pi'$ or to the right of~$b$ in~$\pi$.

We now prove~(v).
This follows immediately by considering all permutations encountered in~$X$ between~$\pi$ to~$\pi'$, by moving $n$ to the left.

It remains to prove~(vii).
It is clear that $f(b,n,L(\pi,b,n))$ is not in~$F_R$, as we could otherwise transpose $b$ and $n$ in $\pi$, reaching a down-neighbor of~$\pi$ in~$X$.
It is also clear that $f(x,c,L(\pi,x,c))$ is not in~$F_R$, as we know that $c>x$ from~(iii), and so we could transpose~$c$ and~$x$, reaching another down-neighbor of~$\pi$ in~$X$.
\end{proof}

With these lemmas in hand, we are now ready to establish a tight lower bound for the minimum degree of quotient graphs~$\cQ_n$.

\begin{theorem}
\label{thm:min-deg}
For every essential lattice congruence $R\in\cC_n^*$, the minimum degree of the quotient graph~$Q_R$ is~$n-1$.
\end{theorem}

Pilaud and Santos proved in~\cite{MR3964495} that for every essential lattice congruence, $Q_R$ is the graph of an $(n-1)$-dimensional polytope.
This in particular implies Theorem~\ref{thm:min-deg}.
Nevertheless, in this paper we provide a purely combinatorial proof of the theorem, with the goal of later improving the estimates in the proof when proving Theorem~\ref{thm:regular}.

\begin{proof}
We prove the theorem by induction on~$n$.
For $n=1$ there is only the trivial lattice congruence for the weak order on~$S_n=\{1\}$, and the corresponding quotient graph is an isolated vertex which indeed has minimum degree~$n-1=0$.
This settles the base case for the induction.

Let $n\geq 2$ and consider an essential lattice congruence~$R\in\cC_n^*$.
The equivalence class~$X$ that contains the identity permutation~$\ide_n$, contains no other permutations by the assumption that $R$ is essential, and so the degree of~$X=\{\ide_n\}$ in~$Q_R$ is exactly~$n-1$ by Lemma~\ref{lem:neighbor}.
By Lemma~\ref{lem:neighbor}, it therefore suffices to show that for every equivalence class~$X$ of $R$ and its minimum~$\pi:=\min(X)$ and maximum~$\rho:=\max(X)$, we have that $\desc(\pi)+\asc(\rho)\geq n-1$.

For this consider the position of the entry~$n$ in both~$\pi$ and~$\rho$.
By the assumption that~$R$ is essential, we know that $f(n-1,n,\emptyset)\notin F_R$, which implies $\ide_n\not\equiv c_{n-1}(\ide_{n-1})$.
Applying Lemma~\ref{lem:rail} shows that $n$ cannot be simultaneously at the rightmost position of~$\pi$ and at the leftmost position of~$\rho$.
Consequently, we are in one of five possible cases:
\begin{enumerate}[label=(\alph*), leftmargin=8mm, noitemsep, topsep=3pt plus 3pt]
\item both $\pi$ and $\rho$ have $n$ at the rightmost position.
\item both $\pi$ and $\rho$ have $n$ at the leftmost position.
\item $\pi$ has $n$ at the rightmost position, and $\rho=\cdots c\,n\,d\cdots$, where $c,d\in[n-1]$.
\item $\rho$ has $n$ at the leftmost position, and $\pi=\cdots a\,n\,b\cdots$, where $a,b\in[n-1]$.
\item $\pi=\cdots a\,n\,b\cdots$ and $\rho=\cdots c\,n\,d\cdots$, where $a,b,c,d\in[n-1]$.
\end{enumerate}
Cases~(c) and~(d) are exactly the ones discussed in Lemma~\ref{lem:minmax-boundary}, and case~(e) is exactly the one discussed in Lemma~\ref{lem:minmax-middle}.
We only prove~(a), (c) and~(e), as the proof of~(b) is analogous to~(a), and the proof of~(d) is analogous to~(c).

By Lemma~\ref{lem:restrict-minmax}, Lemma~\ref{lem:neighbor}, and by induction we know that
\begin{equation}
\label{eq:ind-lb}
\desc(p(\pi))+\asc(p(\rho))\geq n-2.
\end{equation}

First consider case~(a) above.
As $n$ is at the rightmost position in~$\pi$ and~$\rho$, we have
\begin{equation}
\label{eq:daXt}
\desc(\pi)+\asc(\rho)=\desc(p(\pi))+\asc(p(\rho))+1,
\end{equation}
where the +1 comes from the ascent involving~$n$ in~$\rho$.
Combining~\eqref{eq:ind-lb} and~\eqref{eq:daXt} yields
\begin{equation}
\label{eq:dega}
\desc(\pi)+\asc(\rho)\geq (n-2)+1=n-1,
\end{equation}
with equality if and only if~\eqref{eq:ind-lb} holds with equality.

Now consider case~(c) above.
As $n$ is at the rightmost position in~$\pi$ and $n$ is between~$c$ and~$d$ in~$\rho$, we have
\begin{equation}
\label{eq:daXb}
\desc(\pi)+\asc(\rho)=\desc(p(\pi))+\asc(p(\rho))+1-\asc(c\,d),
\end{equation}
where the +1 comes from the ascent~$(c,n)$ in~$\rho$.
By Lemma~\ref{lem:minmax-boundary}, $(c,d)$ is a descent, so $\asc(c\,d)=0$, and hence combining~\eqref{eq:ind-lb} and~\eqref{eq:daXb} yields
\begin{equation}
\label{eq:degc}
\desc(\pi)+\asc(\rho)\geq (n-2)+1=n-1,
\end{equation}
with equality if and only if~\eqref{eq:ind-lb} holds with equality.

Now consider case~(e) above.
In this case we have
\begin{equation}
\label{eq:daXm}
\desc(\pi)+\asc(\rho)=\desc(p(\pi))+1-\desc(a\,b)+\asc(p(\rho))+1-\asc(c\,d),
\end{equation}
where the +1s come from the descent~$(n,b)$ in~$\pi$ and the ascent~$(c,n)$ in~$\rho$, respectively.
We first consider the subcase that~$c=a$ and~$b=d$.
In this case $(a,b)=(c,d)$ is either a descent or an ascent, so in any case $\desc(a\,b)+\asc(c\,d)=1$, and hence combining~\eqref{eq:ind-lb} and~\eqref{eq:daXm} yields
\begin{equation}
\label{eq:dege1}
\desc(\pi)+\asc(\rho)\geq (n-2)+1=n-1,
\end{equation}
with equality if and only if~\eqref{eq:ind-lb} holds with equality.

We now consider the subcase that~$c=a$ and $b\neq d$.
From Lemma~\ref{lem:minmax-middle}~(iv) we obtain that~$c>d$, i.e., $\asc(c\,d)=0$, and hence combining~\eqref{eq:ind-lb} and~\eqref{eq:daXm} yields
\begin{equation}
\label{eq:dege2}
\desc(\pi)+\asc(\rho)\geq (n-2)+2-\desc(a\,b)\geq n-1,
\end{equation}
with equality if and only if~\eqref{eq:ind-lb} holds with equality and~$\desc(a\,b)=1$.

The subcase~$c\neq a$ and $b=d$ is similar, and yields
\begin{equation}
\label{eq:dege3}
\desc(\pi)+\asc(\rho)\geq (n-2)+2-\asc(c\,d)\geq n-1,
\end{equation}
with equality if and only if~\eqref{eq:ind-lb} holds with equality and~$\asc(c\,d)=1$.

It remains to consider the subcase $c\neq a$ and $b\neq d$.
From Lemma~\ref{lem:minmax-middle}~(iii) and~(iv) we know that~$a<b$ and~$c>d$, i.e., $\desc(a\,b)=0$ and $\asc(c\,d)=0$, and hence combining~\eqref{eq:ind-lb} and~\eqref{eq:daXm} yields
\begin{equation}
\label{eq:dege4}
\desc(\pi)+\asc(\rho)\geq (n-2)+2=n>n-1.
\end{equation}
This completes the proof of the theorem.
\end{proof}

\begin{wrapfigure}{r}{0.4\textwidth}
\flushright
\includegraphics{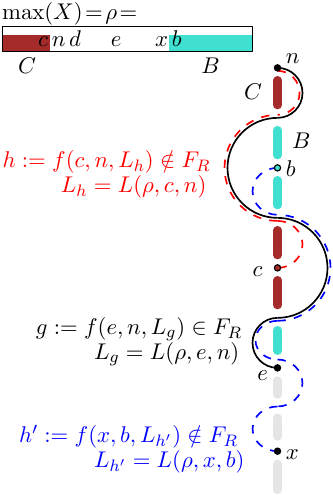}
\captionsetup{width=.85\linewidth}
\caption{Illustration of the proof of Theorem~\ref{thm:regular}.}
\iffull
\vspace{-16mm}
\fi
\ifnotfull
\vspace{-8mm}
\fi
\label{fig:regular}
\end{wrapfigure}

We are now in position to prove the main result of this section, a characterization of regular quotient graphs via their arc diagram.

\begin{theorem}
\label{thm:regular}
The regular quotient graphs~$\cR_n$ are obtained from exactly those lattice congruences~$\cC_n^*$ that have a simple reduced arc diagram.
\end{theorem}

After this paper was published, we learnt that Theorem~\ref{thm:regular} can be derived from Theorem~1.13 in~\cite{demonet_et_al_2018}, which characterizes regular quotient graphs in terms of the minimal sets of join-irreducible elements that need to be contracted to generate the congruence.

\begin{proof}
By Lemma~\ref{lem:non-simple}, if the reduced arc diagram of a lattice congruence~$R\in\cC_n^*$ is not simple, then the quotient graph~$Q_R$ is not regular.
In the following we will prove the converse, that if the reduced arc diagram of~$R$ is simple, then the quotient graph $Q_R$ is $(n-1)$-regular.
We argue by induction on~$n$, using that by Lemma~\ref{lem:restrict-fences}, the arc diagram of the restriction of a lattice congruence is obtained by removing the highest point labeled~$n$ and all arcs incident with it.
In particular, removing the highest point of a simple reduced arc diagram produces another simple reduced arc diagram.

For the induction proof we closely follow the proof of Theorem~\ref{thm:min-deg} given before, and show that all inequalities in that proof are actually tight if we add the assumption of a simple reduced arc diagram.
So consider a lattice congruence~$R\in\cC_n^*$ with a simple reduced arc diagram, an arbitrary equivalence class~$X$ of~$R$, and let $\pi:=\min(X)$ and $\rho:=\max(X)$.
We aim to prove that $\desc(\pi)+\asc(\rho)=n-1$, assuming by induction that~\eqref{eq:ind-lb} holds with equality, i.e., we have
\begin{equation}
\label{eq:ind-lb-tight}
\desc(p(\pi))+\asc(p(\rho))=n-2.
\end{equation}

We now consider the same cases (a)--(e) as in the proof of Theorem~\ref{thm:min-deg}.
The cases~(a) and~(b) are easy, as~\eqref{eq:dega} holds with equality by~\eqref{eq:ind-lb-tight}.
Similarly, the cases~(c) and~(d) are easy, as~\eqref{eq:degc} holds with equality by~\eqref{eq:ind-lb-tight}.
It remains to consider the case~(e).
The subcase~$c=a$ and~$b=d$ is again easy, as~\eqref{eq:dege1} holds with equality because by~\eqref{eq:ind-lb-tight}.
We consider the remaining three subcases of~(e) and show the following:
\begin{enumerate}[label=(e\arabic*), leftmargin=8mm, noitemsep, topsep=3pt plus 3pt]
\item If $c=a$ and $b\neq d$ we have that $\desc(a\,b)=1$.
From this it follows that~\eqref{eq:dege2} holds with equality by~\eqref{eq:ind-lb-tight}.
\item If $c\neq a$ and $b=d$ we have that $\asc(c\,d)=1$.
From this it follows that~\eqref{eq:dege3} holds with equality by~\eqref{eq:ind-lb-tight}.
\item The subcase $c\neq a$ and $b\neq d$ cannot occur (recall the strict inequality~\eqref{eq:dege4}).
\end{enumerate}
In proving (e1)--(e3), we will use the assumption that the reduced arc diagram is simple.
Note that claims (e1)--(e3) follow immediately from the next two claims:
\begin{enumerate}[label=(e\arabic*'), leftmargin=10mm, noitemsep, topsep=3pt plus 3pt]
\item If $b\neq d$ and the arc diagram is simple, then $b<c$.
\item If $c\neq a$ and the arc diagram is simple, then $c<b$.
\end{enumerate}
Indeed, if $c=a$ and $b\neq d$, then (e1') gives $b<c=a$, showing that $\desc(a\,b)=1$, proving~(e1).
Similarly, if $c\neq a$ and $b=d$, then (e2') gives $c<b=d$, showing that $\asc(c\,d)=1$, proving~(e2).
Lastly, if $c\neq a$ and $b\neq d$, then (e1') and (e2') together give~$b<c$ and~$c<b$, a contradiction, so this case cannot occur.

We begin proving~(e1'); see Figure~\ref{fig:regular}.
By Lemma~\ref{lem:minmax-middle}~(iv), $b$ is to the right of~$d$ in~$\rho$.
Let $e$ be the maximum entry between~$n$ and~$b$ in~$\rho$, and let $x$ be the entry directly left of~$b$.
It may happen that $e=d$, or $e=x$ or both, but this is irrelevant.
In fact, the entry~$d$ will not play any role in our further arguments.
We clearly have $e\geq x$.
Applying Lemma~\ref{lem:minmax-middle}~(iv), we obtain that $c,b>e$.
Suppose for the sake of contradiction that~$b>c$.
Combining the previous inequalities, we get $x\leq e<c<b<n$, i.e., we have the situation shown in Figure~\ref{fig:regular}.
From Lemma~\ref{lem:minmax-middle}~(vi), we obtain that the fence $g:=f(e,n,L_g)$ with $L_g:=L(\rho,e,n)$ is in~$F_R$.
We let $C$ denote the set of values that are strictly larger than~$e$ and not to the right of~$c$ in~$\rho$.
Similarly, we let $B$ denote the set of values that are strictly larger than~$e$ and not to the left of~$b$ in~$\rho$.
By these definitions and the maximal choice of~$e$, we get $L_g=C$ and $\ol{L_g}:=\left]e,n\right[\setminus L_g=B$. 

As $c\in L_g$ and $b\in\ol{L_g}$, the arc corresponding to the fence~$g$ that connects~$e$ with~$n$ has~$c$ on its left and~$b$ on its right, i.e., this arc is not simple.
It follows that this arc cannot be in the reduced arc diagram of~$R$.
This means there must be another fence~$g'=f(u,v,\left]u,v\right[\cap L_g)$, $e\leq u<v\leq n$, represented by a simple arc, that forces $g$ in the forcing order.
Clearly, as this arc is simple, we have that
\begin{equation}
\label{eq:ints}
\left]u,v\right[\cap L_g=\emptyset \quad \text{or} \quad \left]u,v\right[\cap L_g=\left]u,v\right[,
\end{equation}
i.e., $]u,v[$ is an interval of consecutive numbers from~$B$ or~$C$, respectively.

From Lemma~\ref{lem:minmax-middle}~(viii), we also know that the fences $h:=f(c,n,L_h)$ with $L_h:=L(\rho,c,n)$ and $h':=f(x,b,L_{h'})$ with $L_{h'}:=L(\rho,x,b)$ are \emph{not} in~$F_R$.
Observe that $L_h=\left]c,n\right[\cap L_g$ and $L_{h'}\cap\left]e,b\right[=L_g\cap\left]e,b\right[$, i.e., the arcs corresponding to the fences~$h$ and $h'$ pass to the left and right of the points in the intervals~$]c,n[$ or~$]e,b[$, respectively, exactly in the same way as the arc corresponding to the fence~$g$; see Figure~\ref{fig:regular}.
It follows that the interval~$[u,v]$ cannot be contained in the interval~$[e,b]$, as otherwise $g'$ would force~$h'$ in the forcing order, and we know that $h'\notin F_R$.
Similarly, it follows that the interval~$[u,v]$ cannot be contained in the interval~$[c,n]$, as otherwise $g'$ would force~$h$ in the forcing order, and we know that $h\notin F_R$.
We conclude that~$u<c$ and~$v>b$.
This however, would mean that $c$ is contained in the interval~$\left]u,v\right[\cap L_g$, but $b$ is not (as $b\notin L_g$), so none of the two conditions in~\eqref{eq:ints} can hold, which means that the fence~$g'$ cannot exist.
(In other words, the arc corresponding to~$g'$ would also have to be non-simple so that~$g'$ could force~$g$.)
We arrive at a contradiction to the assumption~$b>c$.
This completes the proof of~(e1').

The proof of~(e2') is analogous to the proof of~(e1'), and uses Lemma~\ref{lem:minmax-middle}~(iii) instead of~(iv), (v) instead of~(vi), and (vii) instead of~(viii).
We omit the details.
\end{proof}

\begin{wrapfigure}{r}{0.45\textwidth}
\flushright
\includegraphics{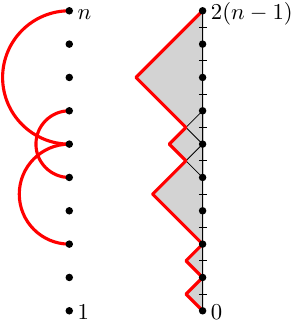}
\captionsetup{width=.92\linewidth}
\caption{Bijection between non-nesting arc diagrams on $n$ points and Dyck paths with $2(n-1)$ steps.}
\vspace{-4mm}
\label{fig:cat}
\end{wrapfigure}

From Theorem~\ref{thm:regular}, we obtain the following corollary.

\begin{corollary}
\label{cor:count-Rn}
The number of regular quotient graphs~$\cR_n$ is $|\cR_n| = C_{n-1}^2$, where $C_n$ is the $n$th Catalan number $C_n:=\frac{1}{n+1}\binom{2n}{n}$.
\end{corollary}

\begin{proof}
By Theorem~\ref{thm:regular}, we need to count simple reduced arc diagrams on $n$ points.
Clearly, arcs passing to the left of the points are independent from arcs passing to the right of the points, so the result is proved by showing that diagrams where all arcs pass on the same side are counted by the Catalan numbers~$C_{n-1}$.
Note that the arcs are all simple, so no arc connects two consecutive points.
Circular arcs in such a diagram are non-nesting, by the assumption that the diagram is reduced, as nested arcs correspond to fences that are comparable in the forcing order.
A bijection between such non-nesting circular arc diagrams on $n$ points and Dyck paths with~$2(n-1)$ steps is illustrated in Figure~\ref{fig:cat}.
\end{proof}

\subsection{Maximum degree}

The next theorem establishes an exact formula for the maximum degrees of quotient graphs.

\begin{theorem}
\label{thm:max-deg}
For every lattice congruence $R\in\cC_n$, the maximum degree of the quotient graph~$Q_R$ is at most $2n-\lceil 2\sqrt{n}\rceil$.
Moreover, there is a lattice congruence with a vertex of this degree.
\end{theorem}

For proving Theorem~\ref{thm:max-deg}, we need the following variant of the famous Erd\H{o}s-Szekeres theorem.

\begin{lemma}
\label{lem:ES+}
Consider a sequence of distinct integers of length~$n$, and let $r$ and $s$ be the length of the longest monotonically increasing and decreasing subsequences, respectively.
Then we have $r+s\geq \lceil 2\sqrt{n}\rceil$.
\end{lemma}

The Erd\H{o}s-Szekeres theorem is usually stated in the slightly weaker form that one of~$r$ \emph{or} $s$ is at least $\lceil\sqrt{n}\rceil$.
The proof of our lemma follows Seidenberg's proof~\cite{MR106189} (see also~\cite{MR1380525}).

\begin{proof}
Let $x_1,\ldots,x_n$ be the sequence we consider.
For $i=1,\ldots,n$, let $a_i$ and $b_i$ be the lengths of the longest increasing or decreasing subsequences ending with $x_i$.
Note that for $1\leq i<j\leq n$ we either have $x_i<x_j$, and then we know that $a_i<a_j$, or we have $x_i>x_j$, and then we know that $b_i<b_j$.
Consequently, all pairs $(a_i,b_i)$ must be distinct, and we have $1\leq a_i\leq r$ and $1\leq b_i\leq s$, implying that $n\leq rs$.
From the arithmetic/geometric mean inequality we obtain $r+s=2(r+s)/2\geq 2\sqrt{rs}\geq 2\sqrt{n}$.
As~$r$ and~$s$ must be integers, this implies the lower bound $r+s\geq \lceil 2\sqrt{n}\rceil$.
\end{proof}

\begin{proof}[Proof of Theorem~\ref{thm:max-deg}]
Consider any permutation~$\rho$ in the weak order on~$S_n$, and consider another permutation $\pi<\rho$ in its downset.
Consider the longest monotonically decreasing subsequence of~$\rho$, and let $r$ denote its length.
Note that $\rho$ has at most $n-1-(r-1)=n-r$ ascents, regardless of the values between the elements of the subsequence.
Similarly, consider the longest monotonically increasing subsequence of~$\rho$, and let $s$ denote its length.
Observe that the elements of this subsequence appear in the same relative order in~$\pi$, so $\pi$ has at most $n-1-(s-1)=n-s$ descents, regardless of the values between the elements of the subsequence.
Overall, we have $\desc(\pi)+\asc(\rho)\leq 2n-(r+s)$.
Applying Lemma~\ref{lem:neighbor} and Lemma~\ref{lem:ES+} completes the proof of the upper bound in the theorem.

\begin{figure}
\includegraphics{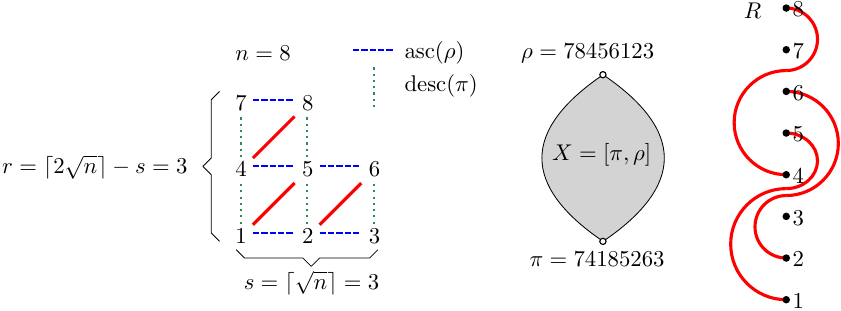}
\caption{Illustration of the proof of Theorem~\ref{thm:max-deg}.
The right part of the figure shows the reduced arc diagram of the lattice congruence~$R$.}
\label{fig:erdos}
\end{figure}

It remains to construct a lattice congruence~$R\in\cC_n$ that has an equivalence class~$X$ with $\desc(\min(X))+\asc(\max(X))=2n-\lceil 2\sqrt{n}\rceil$.
This construction is illustrated in Figure~\ref{fig:erdos}.
Fill the numbers $1,2,\ldots,n$ into a table with $s:=\lceil\sqrt{n}\rceil$ columns, row by row from bottom to top, and from left to right in each row.
The topmost row may not be filled completely.
It can be checked that the number of rows~$r$ of the table is $r=\lceil 2\sqrt{n}\rceil-s$.
Now consider the permutation~$\pi$ obtained by reading the columns of the table from left to right, and from top to bottom in each column.
It satisfies $\desc(\pi)=(n-1)-(s-1)=n-s$.
Also consider the permutation~$\rho$ obtained by reading the rows of the table from top to bottom, and from left to right in each row.
It satisfies $\asc(\rho)=(n-1)-(r-1)=n-r$.
We now construct a lattice congruence~$R$ that has an equivalence class~$X$ with $\min(X)=\pi$ and $\max(X)=\rho$.
From Lemma~\ref{lem:neighbor}, we then obtain that the degree of~$X$ in the quotient graph~$Q_R$ is $\desc(\pi)+\asc(\rho)=2n-(r+s)=2n-\lceil 2\sqrt{n}\rceil$.

We construct~$R$ by specifying a set of fences in the forcing order, and then take the downset of all those fences as~$F_R$.
The fences are constructed as follows:
For each pair of numbers~$a$ and~$b$ in our table where~$b$ is one row above and one column to the right of~$a$, we let~$L$ be the set of all numbers left of~$b$ in the same row as~$b$, and we add the fence~$f(a,b,L)$.
Now $F_R$ is obtained by taking the downset of all those fences in the forcing order.
Observe that as our initial fences~$f(a,b,L)$ all satisfy $b-a=s+1$, all fences~$f(a,b,L)$ in~$F_R$ satisfy $b-a\geq s+1$.
Using the definition of fences and Theorem~\ref{thm:reading}, it can be verified directly that~$\rho$ and~$\pi$ belong to the same equivalence class.
To see that $\pi$ is the minimal element of its equivalence class, note the all descents of~$\pi$ have difference~$s$, so none of these fences is in~$R$, meaning that none of the edges leading to a down-neighbor of~$\pi$ is a bar.
Similarly, to see that $\rho$ is the maximal element of its equivalence class, note that all ascents of~$\rho$ have difference~1, so none of these fences is in~$R$, meaning that none of the edges leading to an up-neighbor of~$\rho$ is a bar.
\end{proof}

\subsection{Vertex-transitive quotient graphs}
\label{sec:count-Vn}

It turns out that all vertex-transitive quotient graphs~$\cV_n$ and~$\cV_n'$ can be characterized and counted precisely via weighted integer compositions and partitions, respectively; see Theorems~\ref{thm:comp} and~\ref{thm:part} and Corollaries~\ref{cor:count-Vn} and~\ref{cor:count-Vn'} below.
As $\cV_n\seq\cR_n$, by Theorem~\ref{thm:regular} we only need to consider simple reduced arc diagrams as candidates for vertex-transitive quotient graphs.
However, as we shall see, we will have to impose further restrictions on the diagram.
Specifically, we refer to an arc corresponding to a fence~$f(a,b,L)$ with $L=\emptyset$ as a \emph{left arc}, and with~$L=\left]a,b\right[$ as a \emph{right arc}.
Also, we say that an arc connecting two points~$s-1$ and~$s+1$, $s\in[2,n-1]$, is \emph{short}.
Moreover, we say that the reduced arc diagram is \emph{empty}, if it contains no arcs.
We will see that all reduced arc diagrams that yield vertex-transitive graphs are suitable concatenations of smaller diagrams that are either empty or contain only short left or right arcs.

The \emph{Cartesian product} of two graphs $G=(V,E)$ and $H=(W,F)$, denoted $G\cprod H$, is the graph with vertex set $V\times W$ and edges connecting $(v,w)$ with $(v',w')$ whenever $v=v'$ and $(w,w')$ is an edge in~$F$, or $w=w'$ and $(v,v')$ is an edge in~$E$.
We write $G\simeq H$ if $G$ and $H$ are isomorphic graphs.
We say that a graph is \emph{prime} if it is not a Cartesian product of two graphs with fewer vertices each.
The following lemma captures a few simple observations that we will need later.

\begin{lemma}[{\cite[page~29+Corollary~4.16+Theorem~4.19]{MR1788124}}]
\label{lem:cart-sim}
The following statements hold for arbitrary connected graphs $G,G',H,H'$:
\begin{enumerate}[label=(\roman*),noitemsep,parsep=0ex,leftmargin=4.5ex]
\item We have $G\cprod H\simeq H\cprod G$.
\item If $G\cprod H\simeq G'\cprod H'$ and both $H$ and~$H'$ are prime, then we have $G\simeq G'$ and $H\simeq H'$, or $G\simeq H'$ and $H\simeq G'$.
\item $G$ and $H$ are vertex-transitive, if and only if $G\cprod H$ is vertex-transitive.
\end{enumerate}
\end{lemma}

Consider a lattice congruence $R\in\cC_n^*$ such that $F_R$ contains two fences~$f(s-1,s+1,\emptyset)$ and~$f(s-1,s+1,\{s\})$ for some~$s\in[2,n-1]$.
Note that any essential fence of the form~$f(a,b,L)$, with $a\in[1,s]$, $b\in[s,n]$ and~$L\seq\left]a,b\right[$ is in the downset of one of the these two fences in the forcing order.
In other words, the reduced arc diagram of~$R$ contains no arc from a point in~$[1,s]$ to a point in~$[s,n]$, except the short left arc and short right arc that connect the points~$s-1$ and~$s+1$.
Moreover, by Lemma~\ref{lem:dim}, the quotient graph~$Q_R$ is obtained as the Cartesian product of the quotient graphs of the two lattice congruences~$A\in\cC_s$ and~$B\in\cC_{n+1-s}$ whose reduced arc diagrams contain exactly the arcs of the reduced arc diagram of~$R$ restricted to the intervals~$[1,s]$ and~$[s,n]$, respectively.
We say that in the reduced arc diagram of~$R$, the short left arc and right arc that connect the points~$s-1$ and~$s+1$ form a \emph{loop centered at $s$}, and we say that the reduced arc diagram of~$R$ is the \emph{product} of the reduced arc diagrams of~$A$ and~$B$.
In this way, the product of two reduced arc diagrams is obtained by gluing together their endpoints, and placing a loop centered at the gluing point.

With slight abuse of notation, we use $S_n$ to also denote the cover graph of the weak order on~$S_n$.
The 5-cycle~$C_5$ is obtained as the quotient graph for the lattice congruence~$R\in\cC_3^*$ given either by $F_R=\{f(1,3,\emptyset)\}$, or by $F_R=\{f(1,3,\{2\})\}$.
Clearly, both~$S_n$ and~$C_5$ are vertex-transitive, and the reduced arc diagram of the former has $n$ points and is empty, and the reduced arc diagram of the latter has 3 points and either one short left arc or one short right arc that connects the first with the third point.
In the following we argue that all vertex-transitive quotient graphs~$\cV_n$ have arc diagrams that are products of these two basic diagrams.
We first rule out any other arc diagrams as candidates for giving a vertex-transitive quotient graph.

Recall that the quotient graph~$Q_R$ has as vertices all equivalence classes of~$R$, and an edge between any two classes~$X$ and~$Y$ that contain a pair of permutations differing in an adjacent transposition.


\begin{lemma}
\label{lem:non-empty}
Let $n \geq 4$, and let $R \in \cC_n^*$ be a lattice congruence whose reduced arc diagram is simple and has no loops.
If the reduced arc diagram is not empty, then $Q_R$ is not vertex-transitive.
\end{lemma}

\begin{proof}
Given a permutation~$\pi\in S_n$ and four distinct entries $a,b,c,d\in[n]$ with $a<b$ and $c<d$ such that~$\pi$ is incident with an $(a,b)$-edge and a $(c,d)$-edge in the cover graph of the weak order on~$S_n$, then $\pi$ forms a 4-cycle in this graph, given by all four permutations obtained from~$\pi$ by transposing $a$ with~$b$, and $c$ with~$d$ in all possible ways.
We denote this 4-cycle by~$C(\pi,(a,b),(c,d))$.
Similarly, given $\pi$ and three distinct entries $a,b,c\in[n]$ with $a<b<c$ such that~$\pi$ is incident with an $(x,y)$-edge and an $(x,z)$-edge, where $\{x,y,z\}=\{a,b,c\}$, then $\pi$ forms a 6-cycle in the cover graph, given by all six permutations obtained from~$\pi$ by permutating $a,b,c$ in all possible ways.
We denote this 6-cycle by~$C(\pi,(a,b,c))$.
Let $L$ denote the set of all entries to the left of all of~$a,b,c$ in~$\pi$.
The 6-cycle $C(\pi,(a,b,c))$ has two edges belonging to the fence~$f(a,b,L\cap\left]a,b\right[)$ and two edges belonging to the fence~$f(b,c,L\cap\left]b,c\right[)$, and we abbreviate these edge sets by~$E_{12}$ and~$E_{23}$, respectively.
It also has exactly one edge belonging to the fence $f(a,c,L\cap\left]a,c\right[)$ and one edge belonging to the fence $f(a,c,L\cap\left]a,c\right[\cup\{b\})$, and we abbreviate these edge sets by~$E_{13\emptyset}$ and $E_{132}$, respectively.
These edge sets of the 4-cycles and 6-cycles mentioned before capture how the type~i and type~ii forcing constraints (recall Figure~\ref{fig:forcing}) act on those cycles.
In the following arguments, we have to distinguish carefully between cycles in the weak order on~$S_n$, and cycles in the quotient graph~$Q_R$.
In particular, a 6-cycle in the weak order may result in a 6-, 5-, or 4-cycle in~$Q_R$, or collapse to a single edge or vertex in~$Q_R$, depending on which of the four aforementioned edges are bars.

By Theorem~\ref{thm:regular}, all vertices in the quotient graph~$Q_R$ have degree~$n-1$.
The strategy of our proof is to consider two particular vertices in the graph, and each of the $\binom{n-1}{2}$ pairs of edges incident with each of those vertices.
Every such pair of edges defines a 4-, 5-, or 6-cycle in~$Q_R$ containing these two edges.
In the corresponding quotientope, these cycles bound the 2-dimensional faces incident to that vertex.
We will show that the number of 4-cycles and $\{5,6\}$-cycles incident to the two vertices is different, implying that the graph is not vertex-transitive.
One of the two vertices we consider is the equivalence class that contains only the identity permutation~$\ide_n$.
The $n-1$ edges incident with it are $(i,i+1)$-edges for $i=1,\ldots,n-1$ (recall Lemma~\ref{lem:neighbor} and that $R$ is assumed to be essential).
There are $n-2$ pairs of an $(i,i+1)$-edge and an $(i+1,i+2)$-edge, and every such pair of edges defines either a 5-cycle or a 6-cycle in~$Q_R$:
Indeed, the edges in~$E_{12}$ and~$E_{23}$ of the 6-cycle $C(\ide_n,(i,i+1,i+2))$ are not bars, as $R$ is essential.
Moreover, by the assumption that the diagram of~$R$ contains no loops, at most one of the fences $f(i,i+2,\emptyset)$ or $f(i,i+2,\{i+1\})$ is in~$F_R$, so at most one of the edges in~$E_{13\emptyset}$ or~$E_{132}$ of the 6-cycle is a bar.
It follows that the corresponding cycle in~$Q_R$ is a 5-cycle or a 6-cycle.
The remaining $\binom{n-1}{2}-(n-2)=\binom{n-2}{2}$ pairs of an $(i,i+1)$-edge and a $(j,j+1)$-edge, $j>i+1$, incident with $\ide_n$ all form a 4-cycle in~$Q_R$:
Indeed, none of the edges of the 4-cycle $C(\ide_n,(i,i+1),(j,j+1))$ are bars, as $R$ is essential.

In the remainder of this proof we identify any arc in the diagram of~$R$ with the fence in the downset~$F_R$ of the forcing order that it represents.
An arc being in the diagram means that the corresponding fence is contracted, i.e., its edges are bars, meaning that the permutations that are the endpoints of such a bar are in the same equivalence class.
Conversely, an arc not being in the diagram means that the corresponding fence is not contracted, i.e., its edges are not bars.

As the reduced arc diagram of~$R$ is not empty, we consider the arc~$f(a,b,L)$ incident to the highest point.
As all arcs are simple, we may assume by symmetry that it is a left arc, i.e., $L=\emptyset$, and if there is also a right arc incident to this point, then we may assume that the left arc is at least as long as the right arc.

\begin{enumerate}[label=(\roman*), leftmargin=8mm, noitemsep, topsep=3pt plus 3pt]
\item The left arc $f(a,b,\emptyset)$ is in the diagram.
\item The endpoints of all arcs are below or at point~$b$.
\item If there is a right arc ending at point~$b$, then its starting point~$a'$ satisfies $a'\geq a$.
\item No two arcs in the diagram are nested, by the assumption that the diagram is reduced, as nested arcs correspond to fences that are comparable in the forcing order.
In particular, no left arc $f(a',b,\emptyset)$, $a'>a$, is in the diagram.
\item All arcs are simple, in particular, no arc connects two consecutive points, as $R$ is assumed to be essential.
\end{enumerate}

We define the sequences $A:=(1,\ldots,a-1)$, $B:=(a+2,\ldots,b-1)$, and $C:=(b+1,\ldots,n)$.
The various cases considered in the following proof are illustrated in Figure~\ref{fig:non-empty}.

\begin{figure}
\includegraphics{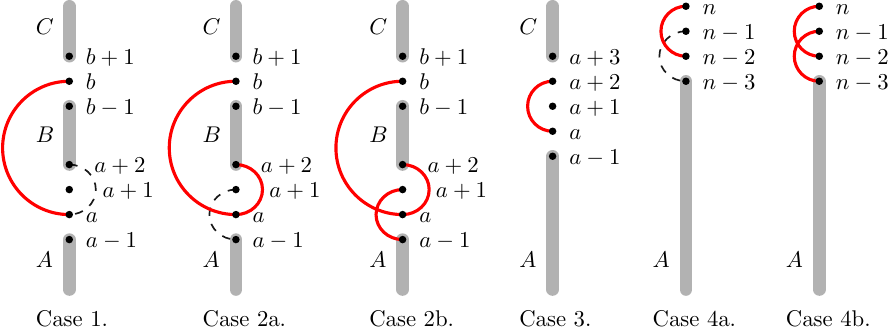}
\caption{Case distinctions in the proof of Lemma~\ref{lem:non-empty}.
Arcs in the diagram are drawn with solid lines, arcs that are \emph{not} in the diagram are indicated by dashed lines.
}
\label{fig:non-empty}
\end{figure}

\textbf{Case 1:} $b-a \geq 3$ and the short right arc $f(a,a+2,\{a+1\})$ is not in the diagram of~$R$.
Consider the equivalence class $X_1$ containing the permutation $\pi_1:=A\,a\,b\,(a+1)\,B\,C$.
It has exactly one descent, namely $(b,a+1)$, and as the left arc $f(a+1,b,\emptyset)$ is not in the diagram by~(iv), we obtain that $\pi_1$ is the minimum of~$X_1$ (recall Lemma~\ref{lem:neighbor}).
Consider the permutation $\rho_1:=b\,A\,a\,(a+1)\,B\,C$, obtained from~$\pi_1$ by transposing the substring~$A\,a$ with~$b$, and so by~(i) and Lemma~\ref{lem:swap-blocks}, $\rho_1$ is also contained in~$X_1$.
Moreover, all $n-2$ ascents in~$\rho_1$ are $(i,i+1)$, for $i\in[n-1]\setminus\{b-1,b\}$, and $(b-1,b+1)$ (if $b<n$), and so by~(ii) and~(v), $\rho_1$ is the maximum of~$X_1$ (recall Lemma~\ref{lem:neighbor}).

\begin{table}[h!]
\caption{Summary of arguments in case~1 in the proof of Lemma~\ref{lem:non-empty}.
Edges that are not bars are marked with (i)--(v), referencing the argument for why that arc is not in the diagram of~$R$.
The edge marked with (*) is not a bar, and the argument is given after the table.
Edges marked with ? are irrelevant for our arguments.
The 6-cycles marked with~[1] are only valid if $b<n$, and those marked with~[2] are only valid if $b<n-1$.
}
\label{tab:case1}
\begin{tabular}{l|l|llll}
\multicolumn{6}{l}{$\rho_1=\max(X_1)=b\,A\,a\,(a+1)\,B\,C$} \\
\multicolumn{6}{l}{$\pi_1=\min(X_1)=A\,a\,b\,(a+1)\,B\,C$} \\ \hline
edges inc.\ only with $\ide_n$ & edges inc.\ only with $X_1$ & & & & \\
$(b-1,b)$ & $(a+1,b)$ & & & & \\
$(b,b+1)$ \hspace{2mm}[1] & $(b-1,b+1)$ \hspace{2mm}[1] & & & & \\ \hline
6-cycles inc.\ only with $\ide_n$ & 6-cycles inc.\ only with $X_1$ & $E_{12}$ & $E_{23}$ & $E_{13\emptyset}$ & $E_{132}$ \\
$C(\ide_n,(b-2,b-1,b))$                 & $C(\pi_1,(a,a+1,b))$                & (v)  & (iv) & ?  & (*)  \\
$C(\ide_n,(b-1,b,b+1))$ \hspace{2mm}[1] & $C(\pi_1,(a+1,a+2,b))$              & (v)  & (iv) & (iv) & ?    \\
$C(\ide_n,(b,b+1,b+2))$ \hspace{2mm}[2] & $C(\pi_1,(b-2,b-1,b+1))$ \hspace{2mm}[1] & (v)  & (ii) & ? & (ii) \\
                                        & $C(\pi_1,(b-1,b+1,b+2))$ \hspace{2mm}[2] & (ii) & (v)  & ? & (ii) \\
\end{tabular}
\end{table}

For the moment we assume that $b<n-1$.
As Table~\ref{tab:case1} shows, there are two edges incident with~$\ide_n$ that are labelled with a transposition that does not appear at any edge incident with~$X_1$.
Conversely, there are two edges incident with~$X_1$ that are labelled with a transposition that does not appear at any edge incident with~$\ide_n$.
Together with the other edges incident with~$X_1$, we obtain three pairs of edges incident only with~$\ide_n$ that define a 6-cycle in the weak order on~$S_n$, and four pairs of edges incident only with~$X_1$ that define a 6-cycle.
All the latter 6-cycles are $\{5,6\}$-cycles in the quotient graph~$Q_R$, showing that the number of $\{5,6\}$-cycles incident with a vertex of~$Q_R$ is by one higher for~$X_1$ than for~$\ide_n$.
The argument that at most one edge from each 6-cycle is a bar, is given at the bottom right of the table, separately for each of the various sets of edges on each cycle.
E.g., for the 6-cycle $C(\pi_1,(a,a+1,b))$, the two edges in~$E_{12}$ belong to the fence $f(a,a+1,\emptyset)$, and the corresponding arc is not in the diagram by~(v).
It remains to argue about case~(*) in the table, i.e., the arc $f(a,b,\{a+1\})$.
This arc is non-simple, and so it is not in the reduced diagram.
Moreover, in the forcing order it can only be forced by a simple arc $f(a',b,\emptyset)$ with $a'\geq a+1$, which is impossible by~(iv), or by the simple short right arc $f(a,a+2,\{a+1\})$, which is impossible by the extra assumption we imposed at the beginning of case~1.
It follows that the arc $f(a,b,\{a+1\})$ is not in the diagram of~$R$.

In the cases $b=n-1$ and $b=n$ one or two of the 6-cycles in the first two columns of Table~\ref{tab:case1} are invalid, but the remaining 6-cycles still result in a surplus of $\{5,6\}$-cycles incident with~$X_1$.

\textbf{Case 2:} $b-a \geq 3$ and the short right arc $f(a,a+2,\{a+1\})$ is in the diagram of~$R$.
Consider the equivalence class~$X_2$ containing the permutation $\pi_2:=A\,(a+1)\,a\,b\,B\,C$.
This permutation has two descents, $(a+1,a)$ and $(b,a+2)$, and it can be checked that $\pi_2=\min(X_2)$.

\textit{Subcase 2a:} We now additionally assume that if $a>1$, then the short left arc $f(a-1,a+1,\emptyset)$ is not in the diagram of~$R$.
Using this assumption, it can be checked that $\rho_2:=\max(X_2)=A\,(a+1)\,b\,B\,C\,a$.
As before, we now consider all transpositions that appear as edge labels at only either $\ide_n$ or $X_2$, and we consider the resulting pairs of transpositions that define a 6-cycle in the weak order on~$S_n$, yielding the 6-cycles shown in Table~\ref{tab:case2a}, where some of them exist only under the extra conditions on~$a$ and~$b$ stated in the table.

\begin{table}[h!]
\caption{Summary of arguments in case~2a.
Some 6-cycles are only valid under the following extra conditions: [1] $b<n$, [2] $b<n-1$, [3] $b>a+3$, [4] $a>1$, [5] $a>2$.
}
\label{tab:case2a}
\begin{tabular}{l|ll}
\multicolumn{3}{l}{$\rho_2=\max(X_2)=A\,(a+1)\,b\,B\,C\,a$} \\
\multicolumn{3}{l}{$\pi_2=\min(X_2)=A\,(a+1)\,a\,b\,B\,C$} \\ \hline
edges inc.\ only with $\ide_n$ & edges inc.\ only with $X$ & \\
$(a-1,a)$ \hspace{2mm}[4] & $(a-1,a+1)$ \hspace{2mm}[4] & \\
$(a+1,a+2)$ & $(a+1,b)$ & \\
$(b-1,b)$ \hspace{2mm}[3] & $(a+2,b)$ \hspace{2mm}[3] & \\
$(b,b+1)$ \hspace{2mm}[1] & $(b-1,b+1)$ \hspace{2mm}[1] & \\ \hline
6-cycles inc.\ only with $\ide_n$ & 6-cycles inc.\ only with $X$ & \\
$C(\ide_n,(a-2,a-1,a))$ \hspace{2mm}[5] & $C(\pi_2,(a-2,a-1,a+1))$ \hspace{2mm}[5] & \\
$C(\ide_n,(a,a+1,a+2))$ & $C(\rho_2,(a-1,a+1,b))$ \hspace{2mm}[4] & (*) \\
$C(\ide_n,(a+1,a+2,a+3))$ \hspace{2mm}[3] & $C(\pi_2,(a,a+1,b))$        & (**) \\
$C(\ide_n,(b-2,b-1,b))$ \hspace{2mm}[3] & $C(\rho_2,(a+1,a+2,b))$ \hspace{2mm}[3] & \\
$C(\ide_n,(b-1,b,b+1))$ \hspace{2mm}\hspace{2mm}[1]+[3] & $C(\pi_2,(a+2,a+3,b))$ \hspace{2mm}[3] & \\
$C(\ide_n,(b,b+1,b+2))$ \hspace{2mm}[2] & $C(\pi_2,(b-2,b-1,b+1))$ \hspace{2mm}[1]+[3] & \\
                                        & $C(\pi_2,(b-1,b+1,b+2))$ \hspace{2mm}[2] & \\
\end{tabular}
\end{table}

Unlike in case~1, where we argued that there are \emph{more} $\{5,6\}$-cycles incident with~$X_1$ than with~$\ide_n$, in case~2 we argue that there are \emph{fewer} $\{5,6\}$-cycles incident with~$X_2$ than with~$\ide_n$.
For this we consider the two 6-cycles marked with~(*) and~(**) in the table, and argue that each of them is contracted to a 4-cycle in the quotient graph~$Q_R$.
Indeed, for the 6-cycle $C(\rho_2,(a-1,a+1,b))$, the edges in~$E_{12}$ are not bars by the assumption that the short left arc $f(a-1,a+1,\emptyset)$ is not in the diagram of~$R$, the edges in~$E_{23}$ are not bars by~(iv), the edge in~$E_{13\emptyset}$ is a bar, as the left arc $f(a-1,b,\emptyset)$ is forced by the left arc $f(a,b,\emptyset)$ in the forcing order, and the edge in~$E_{132}$ is a bar, as the arc $f(a-1,b,\{a+1\})$ is forced by the short right arc $f(a,a+2,\{a+1\})$.
For the 6-cycle $C(\pi_2,(a,a+1,b))$, the edges in~$E_{12}$ are not bars by~(v), the edges in~$E_{23}$ are not bars by~(iv), the edge in $E_{13\emptyset}$ is a bar by~(i), and the edge in~$E_{132}$ is a bar, as the arc $f(a,b,\{a+1\})$ is forced by the short right arc $f(a,a+2,\{a+1\})$.

As one can check from the table, the deficiency of $\{5,6\}$-cycles incident with~$X_2$ compared to~$\ide_n$ continues to hold even when some of the 6-cycles in Table~\ref{tab:case2a} are invalid as a consequence of some or all of the conditions [1]--[5] being violated, as the cycle marked with~(**) is always valid.

\textit{Subcase 2b:} We now assume that $a>1$ and that the short left arc $f(a-1,a+1,\emptyset)$ is in the diagram of~$R$.
Using this assumption, it can be checked that $\rho_2':=\min(X_2)=(a+1)\,b\,B\,C\,A\,a$.
Proceeding similarly to before, we obtain a table that differs from Table~\ref{tab:case2a} exactly by omitting all lines marked~[4] or~[5], i.e., we obtain the same conclusion that there is a deficiency of $\{5,6\}$-cycles incident with~$X_2$ compared to~$\ide_n$.

\textbf{Case 3:} $a<n-2$ and $b=a+2$.
In this case we reconsider the equivalence class~$X_1$ defined in case~1, yielding the following simplified Table~\ref{tab:case3}.

\begin{table}[h!]
\caption{Summary of arguments in case~3.
The 6-cycle marked with~[1] is only valid if $a<n-3$.
}
\label{tab:case3}
\begin{tabular}{l|ll}
\multicolumn{3}{l}{$\rho_1=\max(X_1)=(a+2)\,A\,a\,(a+1)\,C$} \\
\multicolumn{3}{l}{$\pi_1=\min(X_1)=A\,a\,(a+2)\,(a+1)\,C$} \\ \hline
edges inc.\ only with $\ide_n$ & edges inc.\ only with $X_1$ & \\
$(a+2,a+3)$ & $(a+1,a+3)$ & \\ \hline
6-cycles inc.\ only with $\ide_n$ & 6-cycles inc.\ only with $X_1$   & \\
$C(\ide_n,(a+2,a+3,a+4))$ & $C(\rho_1,(a,a+1,a+3)$                   & (*) \\
                          & $C(\pi_1,(a+1,a+3,a+4))$ \hspace{2mm}[1] & \\
\end{tabular}
\end{table}

The cycle $C(\ide_n,(a+2,a+3,a+4))$ is also a 6-cycle incident with~$\ide_n$ in the quotient graph~$Q_R$ by~(ii) and~(v).
We now show that the 6-cycle $C(\rho_1,(a,a+1,a+3))$ marked with~(*) is a 5-cycle incident with $X_1$ in~$Q_R$, which proves that~$X_1$ is incident with one more 5-cycle than~$\ide_n$.
Indeed, the edges in~$E_{12}$ of the marked cycle are not bars by~(v), the edges in~$E_{23}$ are not bars by~(ii), the edge in~$E_{132}$ is not a bar, as the right arc $f(a,a+3,\{a+1,a+2\})$ is not in the diagram of~$R$ by~(ii), and none of the short right arcs $f(a,a+2,\{a+1\})$ or $f(a+1,a+3,\{a+2\})$ that might force it is in the diagram by the assumption that the diagram has no loops, or by~(ii), respectively.
Furthermore, the edge in~$E_{13\emptyset}$ is a bar, as the arc $f(a,a+3,\{a+2\})$ is forced by the short left arc $f(a,a+2,\emptyset)$, which is in the diagram by~(i).

\textbf{Case 4:} $a=n-2$ and $b=a+2=n$.
In this case the short right arc $f(n-2,n,\{n-1\})$ is not in the diagram of~$R$, by the assumption that the diagram contains no loops.

\textit{Subcase 4a:} We now additionally assume that the short left arc $f(n-3,n-1,\emptyset)$ is not in the diagram of~$R$.
In this subcase, we consider two equivalence classes~$X_4$ and~$Y_4$ that are distinct from~$\ide_n$.
The first equivalence class~$X_4$ is the one containing the permutation $\pi_4:=A\,n\,(n-1)\,(n-2)$, and one can check that $\pi_4=\min(X_4)$ and that $\rho_4:=\max(X_4)=n\,A\,(n-1)\,(n-2)$.
The second equivalence class~$Y_4$ contains only a single permutation $\sigma_4=\min(Y_4)=\max(Y_4)=A\,(n-1)\,n\,(n-2)$.
There is only a single 6-cycle incident with only either~$X_4$ or~$Y_4$, namely $C(\rho_4,(n-3,n-2,n-1))$, and one can argue that it is a $\{5,6\}$-cycle in the quotient graph~$Q_R$, implying that there are more $\{5,6\}$-cycles incident with~$X_4$ than with~$Y_4$ in~$Q_R$; see Table~\ref{tab:case4a}.

\begin{table}[h!]
\caption{Summary of arguments in case~4a.
}
\label{tab:case4a}
\begin{tabular}{l|l}
$\rho_4=\max(X_4)=n\,A\,(n-1)\,(n-2)$ \\
$\pi_4=\min(X_4)=A\,n\,(n-1)\,(n-2)$ & $\sigma_4=\min(Y_4)=\max(Y_4)=A\,(n-1)\,n\,(n-2)$ \\ \hline
edges inc.\ only with $X_4$ & edges inc.\ only with $Y_4$ \\
$(n-2,n-1)$ & $(n-2,n)$ \\ \hline
6-cycles inc.\ only with $X_4$ & 6-cycles inc.\ only with $Y_4$ \\
$C(\rho_4,(n-3,n-2,n-1))$ & \\
\end{tabular}
\end{table}

\textit{Subcase 4b:} We now assume that the short left arc $f(n-3,n-1,\emptyset)$ is in the diagram of~$R$.
We consider the equivalence class~$X_4'$ that contains the permutation $\pi_4':=A\,(n-1)\,(n-2)\,n$.
One can check that $\pi_4'=\min(X_4')$ and that $\rho_4':=\max(X_4')=(n-1)\,A\,(n-2)\,n$.
There is only a single 6-cycle incident with only either~$\ide_n$ or~$X_4'$, namely $C(\rho_4',(n-3,n-2,n))$; see Table~\ref{tab:case4b}.
We now argue that this is a $5$-cycle in the quotient graph~$Q_R$, implying that there are more $5$-cycles incident with~$X_4'$ than with~$\ide_n$ in~$Q_R$.
Indeed, the edges in~$E_{12}$ are not bars by~(v), the edges in~$E_{23}$ are not bars, as the short right arc $f(n-2,n,\{n-1\})$ is not in the diagram by the assumption that it has no loops.
Morever, the edge in~$E_{13\emptyset}$ is a bar, as the arc $f(n-3,n,\{n-1\})$ is forced by the short left arc $f(n-3,n-1,\emptyset)$.
Finally, the edge in~$E_{132}$ is not a bar, as the right arc $f(n-3,n,\{n-2,n-1\})$ is not in the reduced diagram by~(iii), and the two short right arcs $f(n-3,n-1,\{n-2\})$ and $f(n-2,n,\{n-1\})$ that may force it in the forcing order are not in the diagram by the assumption of loop-freeness.

\begin{table}[h!]
\caption{Summary of arguments in case~4b.
}
\label{tab:case4b}
\begin{tabular}{l|l}
\multicolumn{2}{l}{$\rho_4'=\max(X_4')=(n-1)\,A\,(n-2)\,n$} \\
\multicolumn{2}{l}{$\pi_4'=\min(X_4')=A\,(n-1)\,(n-2)\,n$} \\ \hline
edges inc.\ only with $\ide_n$ & edges inc.\ only with $X_4'$ \\
$(n-1,n)$ & $(n-2,n)$ \\ \hline
6-cycles inc.\ only with $\ide_n$ & 6-cycles inc.\ only with $X_4'$ \\
                                  & $C(\rho_4',(n-3,n-2,n))$ \\
\end{tabular}
\end{table}

This completes the proof of the lemma.
\end{proof}

\begin{figure}
\includegraphics{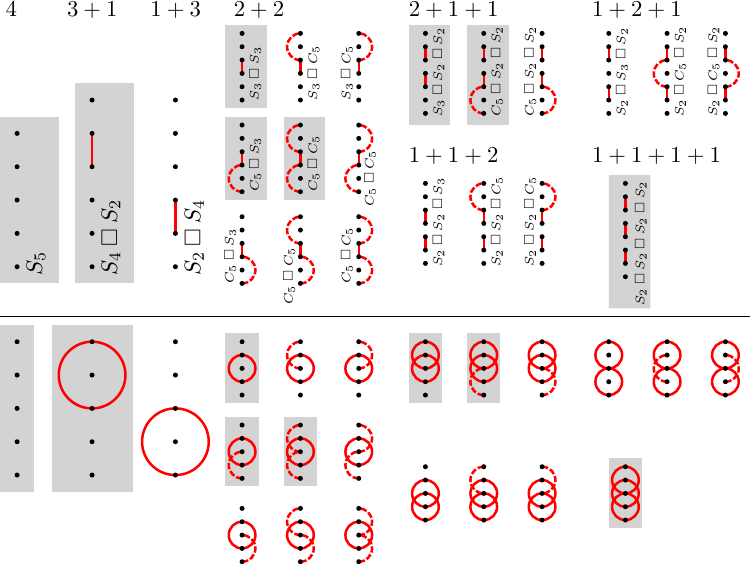}
\caption{Reduced arc diagrams of all 22 vertex-transitive quotient graphs~$\cV_5$ (bottom), plus reduced arc diagrams of corresponding isomorphic quotient graphs in~$\cQ_n$, $n\geq 5$, with the maximal number of non-essential fences (top), plus the corresponding integer compositions of~$4$.
The dashed short arcs correspond to copies of the 5-cycle~$C_5$ in the Cartesian products.
The 8 non-isomorphic quotient graphs~$\cV_5'$ are highlighted, and they correspond to integer partitions.}
\label{fig:V5}
\end{figure}

With Lemma~\ref{lem:non-empty} in hand, we are now ready to characterize vertex-transitive quotient graphs via their arc diagram.

\begin{lemma}
\label{lem:vt}
For $n\geq 2$, every vertex-transitive quotient graph from~$\cV_n$ is a Cartesian product with factors from the set of graphs
\begin{equation}
\label{eq:prime}
\cP:=\{S_2,S_3,S_4,\ldots\}\cup\{C_5\}.
\end{equation}
The corresponding reduced arc diagrams are products of empty diagrams on at least 2~points, and of diagrams on 3~points that have either a short left arc or a short right arc.
\end{lemma}

\begin{proof}
Consider the reduced arc diagram of a lattice congruence $R\in\cV_n$.
By Lemma~\ref{lem:dim}, for any loop in the diagram centered at some point $s\in[2,n-1]$, we may split the diagram into two diagrams on the intervals~$[1,s]$ and~$[s,n]$, and $Q_R$ is the Cartesian product of the quotient graphs of the two lattice congruences defined by the reduced arc diagrams on the two intervals.
We repeat this elimination of loops exhaustively, yielding a factorization of $Q_R\simeq Q_{A_1}\cprod\cdots\cprod Q_{A_p}$ such that the reduced arc diagram of every lattice congruence $A_i\in\cC_{n_i}^*$, $n_i\geq 2$, has no loops.
As $Q_R$ is vertex-transitive, Lemma~\ref{lem:cart-sim}~(iii) yields that all factors~$Q_{A_i}$ must be vertex-transitive.
Therefore, by Lemma~\ref{lem:non-empty}, for any factor with $n_i\geq 4$, we know that the arc diagram of~$A_i$ must be empty, i.e., we have $Q_{A_i}=S_{n_i}$.
For any factor with $n_i=3$, there are exactly three essential lattice congruences yielding a vertex-transitive quotient graph, given by an arc diagram on 3~points that is either empty, or that has a short left arc or a short right arc, and the corresponding graphs are $Q_{A_i}=S_3$ in the first case, and $Q_{A_i}=C_5$ in the latter two cases.
For any factor with $n_i=2$, there is exactly one essential lattice congruence, represented by an empty arc diagram on 2~points, i.e., we have $Q_{A_i}=S_2$.
This proves the lemma.
\end{proof}

Given an integer~$n$, a \emph{composition of $n$} is a way to write $n$ as a sum of positive integers $a_1,\ldots,a_p$, i.e., $n=a_1+\cdots+a_p$.
A \emph{partition of~$n$} is a composition of~$n$ where the summands are sorted decreasingly, i.e., $a_1\geq \cdots \geq a_p$.

\begin{theorem}
\label{thm:comp}
For every $n\geq 2$ and every integer composition $a_1+\cdots+a_p$ of~$n-1$ with exactly $k$ many~2s, there are $3^k$ vertex-transitive quotient graphs in~$\cV_n$, and these graphs are isomorphic to the Cartesian products $G_1\cprod\cdots\cprod G_p$, where $G_i=S_{a_i+1}$ if $a_i\neq 2$ and $G_i\in\{S_3,C_5\}$ if $a_i=2$ for all $i=1,\ldots,p$.
The corresponding reduced arc diagrams are products of empty diagrams on $a_i+1$ points if $a_i\neq 2$, and of diagrams on 3~points that are either empty, or have one short left arc, or one short right arc that connects the first and third point if $a_i=2$.
All of these graphs are distinct, and every graph in~$\cV_n$ arises in this way.
\end{theorem}

\begin{proof}
The proof is illustrated in Figure~\ref{fig:V5}.
We consider an integer composition $a_1+\cdots+a_p$ of~$n-1$ with exactly $k$ many~2s.
For each summand~$a_i\neq 2$, we consider the empty arc diagram on~$a_i+1$ points, and for each summand~$a_i=2$, we consider an arc diagram on 3~points that is either empty, or has one short left arc, or one short right arc.
As the latter case happens $k$ times, we have $3^k$ choices.
Consider the lattice congruences $A_1,\ldots,A_p$ defined by these arc diagrams, and consider the lattice congruence $R\in\cC_n^*$ whose reduced diagram is the product of these diagrams.
By Lemma~\ref{lem:dim}, we have that $Q_R\simeq Q_{A_1}\cprod\cdots\cprod Q_{A_p}$.
Moreover, if $a_i\neq 2$, then we have $Q_{A_i}=S_{a_i+1}$, and if $a_i=2$, then we have $Q_{A_i}\in\{S_3,C_5\}$, i.e., all factors in this product are vertex-transitive.
Applying Lemma~\ref{lem:cart-sim}~(iii), we see that $Q_R$ is vertex-transitive as well.
Clearly, all arc diagrams constructed in this way from integer compositions are distinct, yielding distinct graphs in~$\cV_n$.
By Lemma~\ref{lem:vt}, every graph in~$\cV_n$ arises from such a composition.
\end{proof}

The following corollary is an immediate consequence of Theorem~\ref{thm:comp}.

\begin{corollary}
\label{cor:count-Vn}
Let $c_{n,k}$ denote the number of integer compositions of~$n$ with exactly $k$ many~2s.
For $n\geq 2$, we have $|\cV_n|=\sum_{k\geq 0} 3^k c_{n-1,k}$.
\end{corollary}

Define $a_n:=\sum_{k\geq 0} 3^k c_{n-1,k}$ for $n\geq 2$ and $b_n:=a_{n+1}$ for $n\geq 1$.
The sequence~$b_n$ is OEIS sequence~A052528, and the first few terms are $1, 4, 8, 22, 52, 132, 324, 808, 2000$.
This sequence also has a linear recurrence, namely $b_0=b_1=1$ and $b_n=2b_{n-2}+\sum_{0\leq i\leq n-1}b_i$ for $n\geq 2$.
The generating function is $(1-x)/(2x^3-2x^2-2x+1)$, so the asymptotic growth of~$b_n$ and~$a_n$ is $(1/x_0)^n$, where $x_0$ is the smallest positive root of $2x^3-2x^2-2x+1$, numerically $x_0\approx 0.403032$ and $1/x_0\approx 2.481194$.

\begin{theorem}
\label{thm:part}
For every $n\geq 2$ and every integer partition $a_1+\cdots+a_p$ of~$n-1$ with exactly $k$ many~2s, there are $k+1$ vertex-transitive quotient graphs in~$\cV_n'$, and these graphs are the Cartesian products $G_1\cprod\cdots\cprod G_p$, where $G_i=S_{a_i+1}$ if $a_i\neq 2$ and $G_i\in\{S_3,C_5\}$ if $a_i=2$ for all $i=1,\ldots,p$.
The corresponding reduced arc diagrams are products of empty diagrams on $a_i+1$ points if $a_i\neq 2$, and of $k$ diagrams on 3 points, exactly $j\in\{0,\ldots,k\}$ of which are empty and $k-j$ of which have one short left arc that connects the first and third point, if~$a_i=2$.
All of these graphs are non-isomorphic, and every graph in~$\cV_n'$ arises in this way.
\end{theorem}

The 8 non-isomorphic vertex-transitive graphs for~$n=5$ are highlighted in Figure~\ref{fig:V5}.

\begin{proof}
By Theorem~\ref{thm:comp}, our task is to consider all integer compositions $a_1+\cdots+a_p$ of $n-1$, and among the corresponding Cartesian products~$G_1\cprod\cdots\cprod G_p$, select those which are non-isomorphic graphs.
By Lemma~\ref{lem:cart-sim}~(i), reordering of the factors of any two Cartesian products yields isomorphic graphs, and as all graphs in the set~$\cP$ defined in~\eqref{eq:prime} are prime, Lemma~\ref{lem:cart-sim}~(ii) shows that these reordering operations are the only ones yielding isomorphic graphs.
Consequently, the non-isomorphic quotient graphs can be identified with integer partitions, obtained by sorting the summands of a composition decreasingly, and for a partition with $k$ many 2s, we may choose $j\in\{0,\ldots,k\}$ factors that are 6-cycles~$S_3$, and the remaining $k-j$ factors as 5-cycles~$C_5$.
\end{proof}

\iffull
\begin{figure}
\includegraphics{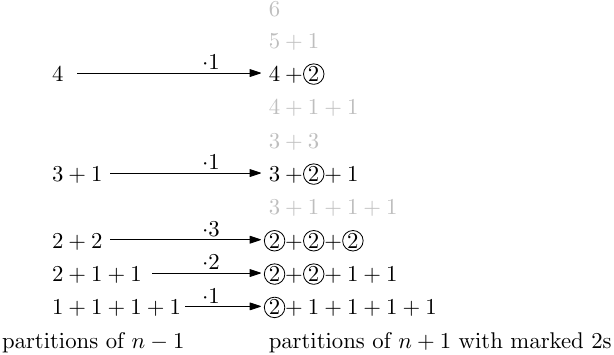}
\caption{Illustration of the proof of Corollary~\ref{cor:count-Vn'} for $n=5$.}
\label{fig:part}
\end{figure}
\fi

\begin{corollary}
\label{cor:count-Vn'}
Let $t_n$ denote the number of 2s in all integer partitions of~$n$.
For $n\geq 2$, we have $|\cV_n'|=t_{n+1}$.
\end{corollary}

By Corollary~\ref{cor:count-Vn'}, the number of vertex-transitive quotient graphs for $n=2,\ldots,10$ is $t_n=1, 3, 4, 8, 11, 19, 26, 41, 56$, respectively, which is OEIS sequence~A024786.
It can be shown that $t_n=e^{\pi\sqrt{2n/3}(1+o(1))}$.

\begin{proof}
By Theorem~\ref{thm:part}, there are exactly $\sum_{k\geq 0} (k+1)p_{n-1,k}$ non-isomorphic vertex-transitive quotient graphs~$\cV_n'$, where $p_{n,k}$ denotes the number of integer partitions of~$n$ with exactly $k$ many~2s.
It remains to show that this sum equals~$t_{n+1}$.
\iffull
This argument is illustrated in Figure~\ref{fig:part}.
\fi
Given any integer partition of~$n-1$ with exactly $k$ many~2s, there are $(k+1)$ ways to insert another marked~2 into this partition, yielding a partition of~$n+1$ with a marked~2.
As all partitions of~$n+1$ with a marked~2 arise in this way, this corresponds exactly to counting the number of~2s in all integer partitions of~$n+1$.
\end{proof}

\begin{remark}
The lattice congruences that yield vertex-transitive quotient graphs described in Theorem~\ref{thm:comp} are precisely the $\delta$-permutree congruences from~\cite{MR3856522} for decorations~$\delta\in\{\noneCirc{},\downCirc{},\upCirc{},\upDownCirc{}\}^n$ such that any $\downCirc{}$ or $\upCirc{}$ is preceded and followed by $\upDownCirc{}$.
\end{remark}

\iffull
\section{Pattern-avoiding permutations and lattice congruences}
\label{sec:pattern}

In the first paper of this series~\cite{MR4391718}, we investigated how Algorithm~J can be used to generate different classes of pattern-avoiding permutations.
In this section, we briefly comment on the relation between pattern-avoiding permutations and lattice congruences.

\subsection{Preliminaries}

Given two permutations~$\pi\in S_n$ and~$\tau\in S_k$, we say that $\pi$ \emph{contains the pattern $\tau$}, if and only if $\pi=a_1\cdots a_n$ contains a subpermutation $a_{i_1}\cdots a_{i_k}$, $i_1<\cdots<i_k$, whose elements are in the same relative order as in~$\tau$.
If $\pi$ does not contain the pattern~$\tau$, then we say that \emph{$\pi$ avoids~$\tau$}.
For example, $\pi=6\colorbox{black!20}{\!35\!}41\colorbox{black!20}{\!2\!}$ contains the pattern $\tau=231$, as witnessed by the highlighted subpermutation.
On the other hand, $\pi=654123$ avoids $\tau=231$.
The patterns we discussed so far are often called \emph{classical} patterns.
In the following we will also need another type of pattern, called a \emph{vincular} pattern.
In a vincular pattern~$\tau$, there is exactly one underlined pair of consecutive entries, with the interpretation that in a match of the pattern~$\tau$ in a permutation~$\pi$, the underlined entries must be matched to adjacent positions in~$\pi$.
For instance, the permutation $\pi=\colorbox{black!20}{\!3\!}1\colorbox{black!20}{\!42\!}$ contains the pattern~$\tau=231$, but it avoids the vincular pattern~$\tau=\ul{23}1$.

We say that a classical pattern~$\tau\in S_k$ is \emph{tame}, if it does not have the largest value~$k$ at the leftmost or rightmost position.
For instance, $\tau=132$ is tame, but $\tau=123$ is not tame.
We say that a vincular pattern~$\tau\in S_k$ is \emph{tame}, if it does not have the largest value~$k$ at the leftmost or rightmost position, and the largest value~$k$ is part of the vincular pair.
For instance, $\tau=3\ul{41}2$ is tame, but $\tau=\ul{12}43$ or $\tau=\ul{41}23$ are not tame.

For any sequence $\tau_1,\ldots,\tau_\ell$ of classical or vincular patterns, we let $S_n(\tau_1,\ldots,\tau_\ell)$ denote the set of all permutations of $[n]$ that avoid each of the patterns $\tau_1,\ldots,\tau_\ell$.

\subsection{Pattern-avoidance and lattice congruences}

In~\cite{MR4391718} we proved the following.

\begin{theorem}[\cite{MR4391718}]
\label{thm:tame}
If $\tau_1,\ldots,\tau_\ell$ are all tame, then for any $n\geq 1$, Algorithm~J generates $S_n(\tau_1,\ldots,\tau_\ell)$ when initialized with the identity permutation.
\end{theorem}

Given Theorems~\ref{thm:tame} and~\ref{thm:lattice}, it is natural to ask: What is the relation between pattern-avoiding permutations and lattice congruences?
Can every lattice congruence of the weak order on~$S_n$ be realized by an avoidance set of tame patterns?
Conversely, does every tame permutation pattern give rise to a lattice congruence?
As we discuss next, the answer to the latter two questions is `no' in general, so pattern-avoiding permutations and lattice congruences are essentially different concepts, except in a few special cases, captured by Theorem~\ref{thm:well} below, and demonstrated by some relevant examples listed after the theorem.

Firstly, it is clear that every lattice congruence of the weak order on~$S_n$ can be described by an avoidance set of patterns of length~$n$, by avoiding all except one permutation from each equivalence class.
Note that in general we will not be able to improve on this approach considerably, as the number of lattice congruences grows double-exponentially with~$n$ (recall Theorem~\ref{thm:count-Qn}), so exponentially many avoidance patterns are needed to describe most lattice congruences.
Moreover, the avoidance patterns resulting from this method will in general not be tame.
For instance, consider the lattice congruence shown in Figure~\ref{fig:cong}, and consider the equivalence class $\{2134,2314,2341\}$.
It contains two patterns that are not tame, so at least one of them has to be avoided following this approach (even though we know that this particular congruence could be described more compactly by avoiding the tame pattern~231).

Conversely, consider the tame pattern~$\tau=2413$, and suppose we want to realize all pattern-avoiding permutations in~$S_5$ as a lattice congruence.
Now consider the permutation~$\pi=25314$, which contains the pattern~$\tau$.
This means $\pi$ must be in the same equivalence class with at least one of the four permutations~$(\pi_1,\pi_2,\pi_3,\pi_4)=(52314,23514,25134,25341)$ that are obtained from~$\pi$ by adjacent transpositions (recall Lemma~\ref{lem:interval}).
This means we have to use at least one of the fences~$f(2,5,\emptyset)$, $f(3,5,\emptyset)$, $f(1,3,\{2\})$, or $f(1,4,\{2,3\})$, respectively, in the forcing order.
However, this forces the pairs of permutations~$(25341,52341)$, $(23541,25341)$, $(52134,52314)$, or $(52314,52341)$, respectively, to also be in the same equivalence class (separately for each pair).
As none of those permutations contains the pattern~$\tau$, we get a contradiction.

We say that a set of vincular patterns~$P$ is \emph{well-behaved}, if each pattern $\tau\in P$ of length~$k$ can be written as $\tau=A\,\ul{k1}\,B$ where $AB$ is a permutation of $\{2,\ldots,k-1\}$, and moreover any vincular pattern obtained by permuting the entries within~$A$ and within~$B$ is also in~$P$.
Note that if $A$ is nonempty, then $\tau=A\,\ul{k1}\,B$ is a tame pattern.
Here are some examples of well-behaved sets of vincular patterns (cf.~\cite[Table~1]{MR4391718}): $P_1=\{2\ul{31}\}$, $P_2=\{2\ul{41}3\}$, $P_2'=\{3\ul{41}2\}$, $P_3=\{3\ul{51}24,3\ul{51}42\}$, $P_3'=\{24\ul{51}3,42\ul{51}3\}$.
The set~$P_1$ is a special case of the family of well-behaved sets of patterns $Q_k:=\big\{\tau\ul{k1}\mid \tau\text{ permutation of }\{2,\ldots,k-1\}\big\}$ for $k\geq 3$ (note that $P_1=Q_3$).
Clearly, this property is preserved under taking unions, so $P_2\cup P_2'$, $P_3\cup P_3'$, or $P_1\cup P_3$ are also well-behaved sets of patterns.

\begin{theorem}
\label{thm:well}
Let $P$ be a well-behaved set of vincular patterns.
For every $\tau=a_1\cdots a_k\in P$, consider the position~$i$ of the largest value~$k$ in~$\tau$, i.e., $a_i\,a_{i+1}=\ul{k1}$, and consider the rewriting rule
\begin{equation}
\label{eq:rewrite}
\_x_1\_\cdots \_x_{i-1}\_x_i x_{i+1}\_x_{i+2}\_\cdots\_ x_k\_ \equiv \_x_1\_\cdots \_x_{i-1}\_x_{i+1} x_i\_x_{i+2}\_\cdots\_ x_k\_,
\end{equation}
where the values $x_1,\ldots,x_k$ appear in the same relative order as in~$\tau$, and which therefore transposes the largest value~$x_i$ and the smallest value~$x_{i+1}$ in this subpermutation.
Combined for all $\tau\in P$, these rewriting rules define a lattice congruence of the weak order on~$S_n$ for any~$n\geq 1$.
Moreover, every equivalence class contains exactly one permutation that avoids every $\tau\in P$, which is the minimum of its equivalence class.
\end{theorem}

Theorem~\ref{thm:well} follows from \cite[Theorem~9.3]{MR2142177}.
Here we provide a short independent proof.

\begin{proof}
We first show that the set of all pairs of permutations $\pi\gtrdot \rho$ in the weak order on~$S_n$ that match one of the rewriting rules~\eqref{eq:rewrite} given by the patterns $\tau\in P$ form a downset of fences in the forcing order.

So suppose we are given $\pi\gtrdot \rho$ in~$S_n$ of the form
\begin{align*}
  \pi  &= \_x_1\_\cdots \_x_{i-1}\_x_i x_{i+1}\_x_{i+2}\_\cdots\_ x_k\_, \\
  \rho &= \_x_1\_\cdots \_x_{i-1}\_x_{i+1} x_i\_x_{i+2}\_\cdots\_ x_k\_,
\end{align*}
where the values $x_1,\ldots,x_k$ appear in the same relative order as in one of the patterns~$\tau\in P$.
Then we have $\pi\equiv\rho$ by~\eqref{eq:rewrite}.
Furthermore, the cover edge $\rho\lessdot \pi$ is contained in the fence~$f(x_{i+1},x_i,L)$, where $L$ is the set of all values between~$x_{i+1}$ and~$x_i$ that appear to the left of~$x_i$ and~$x_{i+1}$ in~$\pi$ and~$\rho$ (in particular, $\{x_1,\ldots,x_{i-1}\}\seq L$).
Let $F_\equiv$ be the union of all those fences, taken over all choices of~$\pi$ and~$\rho$ and all patterns~$\tau\in P$.

Consider a fence $f(a,b,M)$ such that $f(a,b,M)\prec f(x_{i+1},x_i,L)$, i.e., we have $a \leq x_{i+1}$, $b \geq x_i$, $L=M\cap\left]x_{i+1},x_i\right[$, in particular $x_1,\ldots,x_{i-1}\in M$ and $x_{i+2},\ldots,x_k\notin M$.
Consider the two permutations $\pi'\gtrdot\rho'$ defined by
\begin{align*}
  \pi'  &:= x_1\cdots x_{i-1} \,A\, b\,a \, x_{i+2}\cdots x_k \,B, \\
  \rho' &:= x_1\cdots x_{i-1} \,A\, a\,b \, x_{i+2}\cdots x_k \,B,
\end{align*}
where $A$ and $B$ are the increasing sequences of the elements in $M\setminus \{x_1,\ldots,x_{i-1}\}$ and $[n] \setminus (M \cup \{x_{i+2},\ldots,x_{k},a,b\})$, respectively.
Then the edge $\rho'\lessdot \pi'$ is in~$f(a,b,M)$.
Furthermore, as $a \leq x_{i+1}=\min\{x_1,\ldots,x_k\}$ and $b \geq x_i=\max\{x_1,\ldots,x_k\}$, the subpermutation $x_1\cdots x_{i-1} \,b\,a\, x_{i+2}\cdots x_k$ of $\pi'$ is a match of the pattern~$\tau$ and hence $\pi'\equiv\rho'$ by the corresponding rewriting rule, which shows that $f(a,b,M) \in F_\equiv$.
It follows that $F_\equiv$ is a downset of fences in the forcing order, i.e., $F_\equiv$ defines a lattice congruence of the weak order on~$S_n$ (recall Theorem~\ref{thm:reading}).

It remains to show that any two permutations connected by an edge in one of the fences of~$F_\equiv$ are related by one of our rewriting rules.
For this consider a fence~$f(a,b,L)$ from~$F_\equiv$.
By the definition above, there exists a sequence $(x_1,\ldots,x_k)$ of numbers from~$[n]$ and a pattern~$\tau=a_1\cdots a_k\in P$ with the vincular pair at position~$(i,i+1)$, i.e., $a_i\,a_{i+1}=\ul{k1}$, such that $x_1,\ldots,x_k$ appear in the same relative order as in~$\tau$, $a=x_{i+1}$, $b=x_i$, $x_1,\ldots,x_{i-1}\in L$, and $x_{i+2},\ldots,x_k \notin L$.
Hence, for any edge $\rho\lessdot\pi$ in~$f(a,b,L)$, we have that $(x_{i+1},x_i)=(a,b)$ is the pair transposed along this edge, and the values $x_1,\ldots,x_{i-1}$ appear to the left of this pair (not necessarily in this order), and the values $x_{i+1},\ldots,x_k$ appear to the right of this pair (not necessarily in this order), in both~$\rho$ and~$\pi$.
As $P$ is well-behaved, $P$ contains all the patterns obtained from~$\tau$ by permuting $a_1,\ldots,a_{i-1}$ to the left of~$a_i$ and $a_{i+2},\ldots,a_k$ to the right of $a_{i+1}$, implying that $\pi$ matches a pattern from~$P$, and hence there is a rewriting rule witnessing that $\pi\equiv\rho$.

The last part of the theorem follows by interpreting the rewriting rules as a downward orientation of all bars of the lattice congruence~$\equiv$.
Note here that the rule~\eqref{eq:rewrite} removes the inversion $(x_{i+1},x_i)$, so each equivalence class forms a directed acyclic graph, and the unique sink in it, which must be the minimum, is the permutation that avoids all patterns in~$P$.
\end{proof}

Applying Theorem~\ref{thm:well} to the well-behaved sets of patterns mentioned before, $P_1=\{2\ul{31}\}$ yields the Tamari lattice via the sylvester congruence $\_b\_ca\_\equiv \_b\_ac\_$ where $a<b<c$, with 231-avoiding permutations as the minima of the equivalence classes; see Figure~\ref{fig:cong} and note that $S_n(231)=S_n(2\ul{31})$.
More generally, for the set of patterns~$Q_k$ defined before, we obtain the increasing flip lattice on acyclic twists studied in~\cite{MR3741436}.
Moreover, $P_2\cup P_2'=\{2\ul{41}3,3\ul{41}2\}$ yields the rotation lattice on diagonal rectangulations~\cite{MR2871762,MR2914637,MR3878132}, with twisted Baxter permutations as minima.
Lastly, $P_3\cup P_3'=\{3\ul{51}24,3\ul{51}42,24\ul{51}3,42\ul{51}3\}$ yields the lattice on generic rectangulations~\cite{MR2864445}, with 2-clumped permutations as minima.
\fi

\section{Open questions}
\label{sec:open}

We conclude with the following open questions:

\begin{itemize}[itemsep=0ex,parsep=0.5ex,leftmargin=2ex]
\item
We showed that every quotient graph has a Hamilton path, and it is natural to ask whether it also has Hamilton~\emph{cycle} for $n\geq 3$; recall Remark~\ref{rem:cyclic}.
We conjecture that this is the case, and we verified this conjecture with computer help for~$n\leq 5$.
The proof technique described in~\cite{MR1723053} for the associahedron might be applicable for larger classes of quotientopes.

\item
Given our results in Table~\ref{tab:trans}, it seems challenging to characterize the non-isomorphic quotient graphs $\cQ_n'$ by their arc diagrams, and to count them; recall the examples from Figure~\ref{fig:iso}.
In particular, we wonder whether the sequence $|\cQ_n'|$ grows more than exponentially with~$n$, possibly even double-exponentially like $|\cQ_n|$?

\item
It would also be interesting to provide a lower bound on the number of non-isomorphic regular graphs~$\cR_n'$ that improves upon the trivial bound $|\cR_n'|\geq |\cV_n'|$, which comes from number partitions.
In the proof of Theorem~\ref{thm:part}, we can replace any factor~$S_{a_i+1}$, $a_i\geq 3$, coming from a partition with a part~$a_i$ by a prime regular quotient graph.
E.g., for $n=4$ there are 7 non-isomorphic prime regular graphs (10 regular, 4 of which vertex-transitive, 3 of which are products).
This technique could be improved, if we produce non-isomorphic prime regular graphs for larger part sizes~$a_i$.
To this end, it seems worthwile to study the set $\cP_n'\seq \cR_n'$ of non-isomorphic prime regular graphs.
Probably most arc diagrams with a single simple arc are prime graphs, which would show that $|\cP_n'|\geq \Theta(n)$, and this would improve the lower bound for~$|\cR_n'|$.

\item
Another natural direction are colorability properties of quotient graphs.
We found experimentally that all bipartite quotient graphs seem to be characterized as follows:
For any integers $a,b\in[n]$ with $b-a\geq 2$, we let $A(a,b)$ be the set of all arcs in the arc diagram connecting the point~$a$ with the point~$b$.
A quotient graph is bipartite if and only if its arc diagram is a collection of arc sets of the form $A(a_i,b_i)$, where the $[a_i,b_i]$ are non-nesting intervals.
Proving this rigorously would show that bipartite quotient graphs are counted by the Catalan numbers~$C_{n-1}$.

\item
It would also be interesting to investigate which of our results extend to lattice quotients of the weak order on finite Coxeter groups other than the symmetric group~$S_n$ (see \cite{MR3221544,MR3645056}).
\end{itemize}

\section{Acknowledgments}

We thank Jean Cardinal, Vincent Pilaud, and Nathan Reading for several stimulating discussions about lattice congruences of the weak order on~$S_n$.
Figures~\ref{fig:arcs} and~\ref{fig:quotient} in this paper were obtained by modifying and augmenting Figures~6 and~9 from~\cite{MR3964495}, and the original source code for those figures was provided to us by Vincent Pilaud.
We also thank the anonymous reviewer who provided many thoughtful remarks that helped improving this paper.

\bibliographystyle{alpha}
\bibliography{refs}

\end{document}